\documentclass{amsart}

\usepackage{amsmath}
\usepackage{amssymb}
\usepackage{amsfonts}
\usepackage{amsopn}
\usepackage{amsthm}
\usepackage[all]{xy}

\swapnumbers
\theoremstyle{plain}
\newtheorem{thm}{Theorem}[section]
\newtheorem{lem}[thm]{Lemma}
\newtheorem{cor}[thm]{Corollary}
\newtheorem{prop}[thm]{Proposition}

\theoremstyle{definition}
\newtheorem{ntt}[thm]{}
\newtheorem{ex}[thm]{Example}
\newtheorem{rem}[thm]{Remark}
\newtheorem{dfn}[thm]{Definition}

\newcommand{\ba}{\overline{\rule{2.5mm}{0mm}\rule{0mm}{4pt}}} 
\newcommand{\ii}{\mathsf{i}}
\newcommand{\zz}{\mathbb{Z}}
\newcommand{\XX}{\mathfrak{X}}     
\newcommand{\BX}{{\XX_0}}   
\newcommand{\LL}{\mathcal{L}}      

\newcommand{\resCH}{\mathrm{res}_{\scriptstyle\mathrm{CH}}}
\newcommand{\resch}{\mathrm{res}_{\scriptstyle\mathrm{{Ch}}}}
\newcommand{\resk}{\mathrm{res}_{\scriptstyle\mathrm{{K_0}}}}
\newcommand{\resg}{\mathrm{res}_{{{\gamma}}}}
\newcommand{\Fp}{\mathbb{F}_p}     
\newcommand{\mf}{\mathbb{F}_2}
\newcommand{\D}{\mathrm{D}}

\newcommand{\cc}{\mathfrak{c}}     
\newcommand{\simpc}{\mathrm{sc}}
\newcommand{\br}{{\mathrm{Br}}}  
\newcommand{\qform}[1]{{\langle{#1}\rangle}}  

\newcommand{\ra}{\rightarrow}

\newcommand{\T}{{\hat{T}_0}}

\DeclareMathOperator{\ind}{\mathrm{ind}}  
\DeclareMathOperator{\Ch}{\mathrm{Ch}}    
\DeclareMathOperator{\Pic}{\mathrm{Pic}}  
\DeclareMathOperator{\CH}{\mathrm{CH}} 

\DeclareMathOperator{\im}{\mathrm{im}}    
\DeclareMathOperator{\ad}{ad} 		

\newcommand{\PGO}{\mathrm{PGO}^+}      
\newcommand{\SO}{\mathrm{O}^+}
\newcommand{\GL}{\mathrm {GL}}
\newcommand{\Spin}{\mathrm {Spin}}

\newcommand{\C}{\mathcal{C}}   

\title{The $J$-invariant, Tits algebras and triality}

\author{A.~Qu\'eguiner-Mathieu, N.~Semenov, K.~Zainoulline}



\address{(Qu\'eguiner-Mathieu) LAGA - UMR 7539 - Institut Galil\'ee - Universit\'ee Paris 13 - 93430 Villetaneuse - France}
\urladdr{http://www.math.univ-paris13/~queguin/}

\address{(Semenov) 
Institut f\"ur Mathematik - Johannes Gutenberg-Universit\"at - Staudingerweg~9 - 55099 Mainz - Germany}
\urladdr{http://www.staff.uni-mainz.de/semenov/}

\address{(Zainoulline) Department of Mathematics and Statistics - University of Ottawa - 585~King Edward - Ottawa ON K1N6N5 - Canada}
\urladdr{http://www.mathstat.uottawa.ca/{\textasciitilde}kirill/}

\begin{document}

\begin{abstract}
In the present paper we set up a connection between the
indices of the Tits algebras of a simple linear algebraic group $G$
and the degree one parameters of its motivic $J$-invariant. Our main technical tool
are the second Chern class map and Grothendieck's $\gamma$-filtration. 

As an application we recover some known results on the $J$-invariant 
of quadratic forms of small dimension; we describe all possible values of the $J$-invariant of an algebra with
orthogonal involution up to degree $8$ and give explicit examples; we establish several relations between
the $J$-invariant of an algebra $A$ with orthogonal involution
and the $J$-invariant of the corresponding quadratic form over the function field of the Severi-Brauer variety of $A$.
\end{abstract}

\maketitle

\tableofcontents


\section*{Introduction}

The notion of a {\em Tits algebra} was introduced by Jacques Tits in his celebrated
paper on irreducible representations~\cite{Ti71}. 
This invariant of a linear algebraic group $G$
plays a crucial role in the computation of the $K$-theory of
twisted flag varieties by Panin \cite{Pa94} and 
in the statements and proofs of the index reduction formulas by 
Merkurjev, Panin and Wadsworth \cite{MPW}. 
It has important applications to the classification of
linear algebraic groups, and to the study of
the associated homogeneous varieties. 

Another invariant of a linear algebraic group, 
the {\em $J$-invariant},
has been recently defined in \cite{PSZ}. 
It can be viewed as an extension to arbitrary groups of the $J$-invariant of a quadratic form which
was studied during the last decade, notably by 
Karpenko, Merkurjev, Rost and Vishik. 
This invariant describes the motivic behavior of the variety of Borel subgroups of $G$, 
and it consists of an $r$-tuple of integers $J=(j_1,j_2,\ldots,j_r)$ for each torsion prime $p$ of $G$.

The main goal of the present paper
is to set up a connection between the
indices of the Tits algebras of a group $G$
and the degree one parameters of its motivic $J$-invariant. 
The main results are Cor.~\ref{mainineq} and Thm.~\ref{corj0}, which consists of inequalities relating those integers. 
As a crucial ingredient, we use Panin's computation of $K_0(\XX)$, where $\XX$ is the variety of Borel subgroups of
$G$~\cite{Pa94}. The result is obtained using Grothendieck's $\gamma$-filtration on $K_0(\XX)$, and relies on
Lemma~\ref{ci}, which describes Chern classes 
of rational bundles of the first two layers of the $\gamma$-filtered group $K_0(\XX)$. 
 
In the remaining part of the paper, devoted to applications, we always take $p=2$. 
Let $(A,\sigma)$ be a central simple algebra endowed with an involution of orthogonal type and trivial discriminant. 
Its automorphism group is a group of type $\D_n$, and the $J$-invariant in this setting provides a discrete invariant of $(A,\sigma)$, 
which coincides in the split case with the $J$-invariant of the underlying quadratic form. 
As an application of our main result, we give explicit formulas to compute this invariant for algebras of small degree. 
The most interesting case is the degree $8$ case, treated in Thm.~\ref{trial.thm}, where the proof is based on triality, and precisely on its
consequences on Clifford algebras (see~\cite[\S 42.A]{BI}). 

The $J$-invariant is a finite and bounded tuple of integers. Hence, given a fixed prime number $p$, there are finitely many possible values for the $J$-invariant of a group of given type and rank. Moreover, since the image under the Steenrod operation of a rational cycle is rational, it is proven in~\cite[4.12]{PSZ} that some values are impossible (see \S\ref{appendix} below for a precise table). For quadratic forms, this was already noticed by Vishik~\cite[\S 5]{Vi05}, who also checked that these restrictions are the only ones for quadratic forms of small dimension ({\em loc. cit.} Question 5.13). 
As opposed to this, we prove in Cor.~\ref{values.cor} that some values for the $J$-invariant, which are not excluded by the Steenrod operation, do not occur. 
This happens already in degree $8$, and follows from a classical result on algebras with involution~\cite[(8.31)]{BI}, due to Tits and Allen. 

Finally, we study in \S\ref{generic.section} the properties of the quadratic form attached to an orthogonal involution $\sigma$ after generic splitting of the underlying algebra $A$. In particular, when this form belongs to the $s$-th power of the fundamental ideal of the Witt ring, we get interesting consequences on the $J$-invariant of $(A,\sigma)$ in Thm.~\ref{Meta-conj}. 

Finally, we should add that 
Junkins in \cite{Ju11} has successfully applied and extended our results for prime $p=3$ 
to characterize the behaviour of Tits indices of 
exceptional groups of type $\mathrm{E}_6$.

\subsection*{Notations} 

\subsubsection*{Algebraic groups and varieties of Borel subgroups}
We work over a base field $F$ of characteristic different from $2$. 
Let $G_0$ be a split simple linear algebraic group of rank $n$ over $F$.
We fix a split maximal torus $T_0\subset G_0$, and a Borel subgroup $B_0\supset T_0$, and we let $\hat T_0$
be the character group of $T_0$. 
We let $\Pi=\{\alpha_1,\alpha_2,\ldots,\alpha_n\}$ 
be a set of simple roots with respect to $B_0$, and 
$\{\omega_1,\omega_2,\ldots,\omega_n\}$ 
the respective set of fundamental weights, so that 
$\alpha_i^\vee(\omega_j)=\delta_{ij}$. 
The roots and weights are always numbered as in Bourbaki~\cite{Bou}. 

Recall that $\Lambda_r\subset \hat T_0\subset \Lambda_\omega$, where 
$\Lambda_r$ and $\Lambda_\omega$ are the root and weight lattices, respectively. 
The lattice $\hat T_0$ coincides with $\Lambda_r$ (respectively $\Lambda_\omega$) if and only if $G_0$ is adjoint
(resp. simply connected).  

\begin{ntt} 
\label{inner}
Throughout the paper, $G$ denotes a twisted form of $G_0$, and $T\subset G$ is the corresponding maximal torus. 
We always assume that $G$ is an inner twisted form of $G_0$, and even a little bit more, that 
is $G={}_\xi G_0$ for some cocycle $\xi\in Z^1(F,G_0)$. 
Note that this hypothesis need not be satisfied by all groups isogeneous to $G$; for instance, 
it is valid for the simply connected cover $G^{\mathrm{sc}}$ of $G$ if and only if $G$ is a strongly inner
twisted form of $G_0$. 
\end{ntt}

\begin{ntt} 
\label{picard}
We let $\BX$ be the variety of Borel subgroups of $G_0$, or equivalently of its simply connected cover
$G_0^{\mathrm{sc}}$, and $\XX={}_\xi\BX$ the corresponding twisted variety. 
Both varieties are defined over $F$, and they are isomorphic over a separable closure $F_s$ of $F$. 
The Picard group $\Pic(\BX)$ can be computed as follows. 
Since any character $\lambda\in \T$ extends uniquely to $B_0$, 
it defines a line bundle $\LL(\lambda)$ over $\BX$. 
Hence, there is a natural map $\T\ra\Pic(\BX)$, 
which is 
an isomorphism if $G_0$ is simply connected. 
So, we may identify the Picard group 
$\Pic(\BX)$ with the weight lattice $\Lambda_\omega$ (this fact goes back to Chevalley, see also \cite[Prop.~2.2]{MT}). 
\end{ntt} 

\subsubsection*{Algebras with involution} 
We refer to~\cite{BI} for definitions and classical facts on algebras with involution. 
Throughout the paper, $(A,\sigma)$ always stands for a central simple algebra of even degree $2n$, 
endowed with an involution of orthogonal type with trivial discriminant. 
In particular, this implies that the Brauer class $[A]$ of the algebra $A$ is an element of order $2$ of the Brauer group $\br(F)$. 
Because of the discriminant hypothesis, the Clifford algebra of $(A,\sigma)$, endowed with its canonical involution, 
is a direct product 
$(\C(A,\sigma),\underline\sigma)=(\C_+,\sigma_+)\times(\C_-,\sigma_-)$ of two central simple algebras. 
If moreover $n$ is even, the involutions $\sigma_+$ and $\sigma_-$ also are of orthogonal type. 

\begin{ntt}\label{hypinv}
We refer to~\cite[\S6]{BI} for the definition of isotropic and hyperbolic involutions. 
In particular, recall that $A$ has a hyperbolic involution if and only if it decomposes as $A=M_2(B)$ for some 
central simple algebra $B$ over $F$. 
When this occurs, $A$ has a unique hyperbolic involution $\sigma_0$ up to isomorphism.  
Moreover, $\sigma_0$ has trivial discriminant, and 
if additionally the degree of $A$ is divisible by $4$, then its Clifford algebra has a split component by~\cite[(8.31)]{BI}.  
\end{ntt}  

\begin{ntt} 
\label{cocenter}
The connected component of the automorphism group of $(A,\sigma)$ is denoted 
by $\PGO(A,\sigma)$. Since the involution has trivial discriminant, it is an inner twisted form of $\PGO_{2n}$, and hence satisfies~(\ref{inner}). 
Both groups are adjoint of type $\D_n$. 
We recall from Bourbaki~\cite{Bou} the description of their cocenter $\Lambda_\omega/\Lambda_r$ in terms of the 
fundamental weights, for $n\geq 3$: 

If $n=2m$ is even, then $\Lambda_\omega/\Lambda_r\simeq\zz/2\zz\times\zz/2\zz$, and the three non-trivial elements are the classes of $\omega_1$, $\omega_{2m-1}$ and $\omega_{2m}$. 

If $n=2m+1$ is odd, then $\Lambda_\omega/\Lambda_r\simeq\zz/4\zz$, and the generators are the classes of $\omega_{2m}$ and $\omega_{2m+1}$. 
Moreover, the element of order $2$ is the class of $\omega_1$. 
\end{ntt}

\subsubsection*{Tits algebras} 
Consider the simply connected cover $G_0^{\mathrm{sc}}$ of $G_0$ and the corresponding 
twisted group $G^\simpc$.
We denote by $\Lambda_\omega^+\subset \Lambda_\omega$ the cone of dominant weights. 
Since $G$ is an inner twisted form of $G_0$, for any $\omega\in\Lambda_\omega^+$ the corresponding 
irreducible representation $G_0^{\mathrm{sc}}\ra \GL(V)$, viewed as a representation of $G^{\mathrm{sc}} \times F_s$, descends to an algebra representation 
$G^{\mathrm{sc}}\ra \GL_1(A_\omega)$, where $A_\omega$ is a central simple algebra over $F$, called a Tits algebra of $G$ (cf.~\cite[\S 3,4]{Ti71} or~\cite[\S 27]{BI}).
In particular, to any fundamental weight $\omega_\ell$ corresponds a Tits algebra $A_{\omega_\ell}$
\begin{ntt}
\label{Titsalgebra}
Taking Brauer classes, the assignment $\omega\in \Lambda_\omega^+\mapsto A_\omega$ induces a homomorphism to the Brauer group of $F$ ({loc. cit.}) 
\[\beta\colon\Lambda_\omega/\T\ra {\mathrm{Br}}(F).\] 
\begin{ex}\label{FR}
If $G$ is adjoint of type $\D_n$, that is $G=\PGO(A,\sigma)$, the Tits algebras $A_{\omega_1}$, $A_{\omega_{2m-1}}$ and $A_{\omega_{2m}}$ are respectively the algebra $A$ and the two
components $\C_+$ and $\C_-$ of the Clifford algebra of $(A,\sigma)$ (see~\cite[\S 27.B]{BI}). 
Applying Tits homomorphism, and taking into account the description of $\Lambda_\omega/\T=\Lambda_\omega/\Lambda_r$, we get the so-called fundamental relations~\cite[(9.12)]{BI}
relating their Brauer classes, namely: 

If $n=2m$ is even, that is $\deg(A)\equiv0\mod 4$, then $[\C_+]$ and $[\C_-]$ are of order at most $2$, and 
$[A]+[\C_+]+[\C_-]=0\in\br(F)$. 
In other words, any of those three algebras is Brauer equivalent to the tensor product of the other two. 

If $n=2m+1$ is odd, that is $\deg(A)\equiv2\mod 4$, then $[\C_+]$ and $[\C_-]$ are of order dividing $4$, and 
$[A]=2[\C_+]=2[\C_-]\in\br(F)$. 

\end{ex}
For any $\omega\in\Lambda_\omega$, we denote by $i(\omega)$ the index of the Brauer class $\beta(\bar\omega)$, that is the degree of the underlying division algebra.
For fundamental weights, $i(\omega_\ell)$ is the index of the Tits algebra $A_{\omega_\ell}$. 
\end{ntt} 

\section{Characteristic maps and restriction maps}\label{sec2}

\label{char.section}

\subsection*{Characteristic map for Chow groups}

Let $\CH^*(-)$ be the graded Chow ring of algebraic
cycles modulo rational equivalence. 
Since $\BX$ is smooth projective, the first Chern class induces 
an isomorphism between the Picard group $\Pic(\BX)$ and $\CH^1(\BX)$~\cite[Cor. II.6.16]{Hart}. 
Combining with the isomorphism $\Lambda_\omega\simeq \Pic(\BX)$ of~\ref{picard}, we get an isomorphism, which 
is the simply connected degree $1$ characteristic map:  
\[\cc^{(1)}_\mathrm{sc}\colon\Lambda_\omega\tilde\rightarrow \CH^1(\BX).\]
Hence, the cycles 
$$
h_i:=c_1(\LL(\omega_i)),\quad i=1\ldots n,
$$ 
form a $\zz$-basis of the group $\CH^1(\BX)$. 

\begin{ntt}\label{cod1c} 
In general, the degree $1$ characteristic map is the restriction of this isomorphism 
to the character group of $T_0$, 
\[
\cc^{(1)}\colon \T\subset\Lambda_\omega\to \CH^1(\BX).
\]
Hence, it maps 
$\lambda=\sum_{i=1}^n a_i\omega_i\in\T$, where $a_i\in\zz$, to 
$c_1(\LL(\lambda))=\sum_{i=1}^n a_ih_i$. 
For instance, in the adjoint case, the image of $\cc^{(1)}$ is generated by 
linear combinations
$\sum_j c_{ij}h_j$,
where $c_{ij}=\alpha_i^\vee(\alpha_j)$ 
are the coefficients of the Cartan matrix. 
\end{ntt}

\begin{ex} 
\label{excharmap}
We let $p=2$ and consider the Chow group with coefficients in $\mf$ 
$\Ch^1(\BX)=\CH^1(\BX)\otimes_\zz\mf$. 
Assume $G_0$ is of type $\D_4$. 
Using the simply connected characteristic map, we may identify the degree $1$ Chow group modulo $2$ with the $\mf$-lattice
\[
 \Ch^1(\BX)=\mf h_1\oplus\mf h_2\oplus\mf h_3\oplus \mf h_4
\]
Moreover, numbering roots as in~\cite{Bou}, the Cartan matrix modulo $2$ is given by
$$
\left(\begin{array}{cccc}
0 & 1 & 0 & 0 \\
1 & 0 & 1 & 1 \\
0 & 1 & 0 & 0 \\
0 & 1 & 0 & 0
\end{array}\right)
$$
Therefore, in the adjoint case 
the image of the characteristic map $\cc^{(1)}_{\mathrm {ad}}$ with
$\mathbb{F}_2$-coefficients
is the subgroup 
\[
 \im(\cc^{(1)}_{\mathrm {ad}})=\mf h_2\oplus \mf (h_1+h_3+h_4)\subset\Ch^1(\BX)
\]
In the half-spin case, that is when one of the two weights $\omega_3$, $\omega_4$ is in $\T$, say 
$\omega_3\in \T$, we get
\[
 \im(\cc^{(1)}_{\mathrm {hs}})=\mf h_2\oplus\mf h_3 \oplus\mf (h_1+h_4)\subset\Ch^1(\BX)
\] 
\end{ex}

\begin{ntt} 
\label{defchar}
The degree $1$ characteristic map extends to a characteristic map 
$$\cc \colon \mathop{S^*}(\T)\to \CH^*(\BX),$$
where $\mathop{S^*}(\T)$ 
is the symmetric algebra of $\T$ (see \cite[\S4]{Gr58}, \cite[\S1.5]{De74}). 
Its image $\im(\cc)$ is generated by the elements of codimension one, that is by the image of $\cc^{(1)}$. 
\end{ntt}

\subsection*{Restriction map for Chow groups}

\begin{ntt}
\label{charinres}
Let $G$ and $\xi\in Z^1(F,G_0)$ be as in~\ref{inner}, so that $G={}_\xi G_0$. 
The cocycle $\xi$ induces an identification $\XX\times_F F_s\simeq\BX\times_F F_s$. 
Moreover, since $\BX$ is split, $\CH^*(\BX\times_F F_s)=\CH^*(\BX)$. 
Hence the restriction map can be viewed as a map 
\[
\resCH\colon \CH^*(\XX)\to \CH^*(\XX\times_F F_s)\simeq \CH^*(\BX).
\]
A cycle of $\CH^*(\BX)$ is called rational if it belongs to the image of the restriction.  

In~\cite[Thm.6.4(1)]{KM06}, it is proven that, under the hypothesis~(\ref{inner}), any cycle in the image
of the characteristic map $\cc$ is rational, i.e.
$$\im(\cc)\subset \im(\resCH).$$
(See~\cite[\S 7]{KM06} to compare their $\bar\varphi_G$ with our characteristic map.) 
\end{ntt}

\begin{rem} 
Note that the image of the restriction map does not depend on the choice of $G$ in its isogeny class, while 
the image of the characteristic map does. But the inclusion holds only if $G$ can be obtained from $G_0$ by twisting 
by a cocycle with values in $G_0$. Such a cocycle cannot always be lifted to a cocycle with values 
in a larger group isogeneous to $G_0$. For instance, if $G$ is not a strongly inner form of $G_0$, then there 
might be some non rational cycles in the image of the simply connected characteristic map. 
\end{rem}

For a split group $G_0$, the restriction map is an isomorphism, and this inclusion is strict, except if $H^1(F,G_0)$ is trivial. 
On the other hand, generic torsors are defined as the torsors for which it is an equality: 
\begin{dfn} A cocycle $\xi\in Z^1(F,G_0)$ defining the twisted group $G={}_\xi G_0$ 
is said to be {\em generic}
if any rational cycle is in $\im(\cc)$, so that 
$$\im(\cc)= \im(\resCH).$$
\end{dfn}
Observe that a generic cocycle always exists over some
field extension of $F$ by \cite[Thm.6.4(2)]{KM06}.

\subsection*{Characteristic map for $K_0$ and the Steinberg basis}

\begin{ntt}\label{charK0}
Using the identification between $\Lambda_\omega$ and $\Pic(\BX)$ of~\ref{picard}, 
one also gets a characteristic map for $K_0$ (see \cite[\S2.8]{De74}),  
$$
\cc_K\colon \zz[\T] \to K_0(\BX),
$$
where $\mathbb{Z}[\T]\subset\mathbb{Z}[\Lambda_\omega]$ denotes the integral group ring 
of the character group $\T$. 
Any generator $e^\lambda$, $\lambda\in\T$, maps to the class
of the associated line bundle $[\LL(\lambda)]\in K_0(\BX)$.

Combining a theorem of Pittie~\cite{Pit} (see also~\cite[\S 0]{Pa94}), and Chevalley's description of the representation 
rings of the simply connected cover $G_0^\simpc$ of $G_0$ and its Borel subgroup $B_0^\simpc$, one may check that 
$K_0(\BX)$ is isomorphic to the tensor product $\mathbb{Z}[\Lambda_\omega]\otimes_{\mathbb{Z}[\Lambda_\omega]^W}\zz$. 
That is, the simply-connected characteristic map $$\cc_{K,\simpc}\colon\zz[\Lambda_\omega] \to K_0(\BX)$$ is surjective, and its kernel
is generated by the elements of the augmentation ideal that are invariant under the action of 
the Weyl group $W$.
\end{ntt}

\begin{ntt} \label{K0}
Moreover, Steinberg described in~\cite[\S2]{St75} (see also~\cite[\S12.5]{Pa94})
an explicit basis of $\mathbb{Z}[\Lambda_\omega]$ as a free module over $\mathbb{Z}[\Lambda_\omega]^W$. 
It consists of the weights 
$\rho_w$ defined for any $w$ in the Weyl group $W$ by 
$$
\rho_w=\sum_{\{\alpha_k\in\Pi,\  w^{-1}(\alpha_k)\in\Phi^-\}} w^{-1}(\omega_k),
$$
where $\Phi^-$ denotes the set of negative
roots with respect to $\Pi$. 
Hence, we get that the elements 
$$
g_w:=\cc_{K,\simpc}(e^{\rho_w})=[\LL(\rho_w)],\quad w\in W,
$$ 
form a $\mathbb{Z}$-basis
of $K_0(\BX)$, called the {\it Steinberg basis}.
Note that if $w$ is the reflection $w=s_i$, $1\leq i\leq n$, associated to the root $\alpha_i$, we get 
$$
\rho_{s_i}=\sum_{\{\alpha_k\in\Pi,\  s_i(\alpha_k)\in\Phi^-\}} s_i(\omega_k)=s_i(\omega_i)=
\omega_i-\alpha_i.
$$
\begin{dfn}\label{defspecel} 
The elements of the Steinberg basis
$$
g_i=[\LL(\rho_{s_i})],\quad i=1\ldots n
$$
are called {\em special}. 
\end{dfn}
\end{ntt}

\subsection*{Restriction map for $K_0$ and the Tits algebras}

\begin{ntt}
 \label{Panin}
As we did for Chow groups, we use the identification $\XX\times_F F_s\simeq\BX\times_F F_s$ to view the restriction map for $K_0$ as a morphism 
\[ \resk\colon K_0(\XX)\rightarrow K_0(\BX)=\bigoplus_{w\in W}\zz\cdot g_w.\]
By Panin's theorem~\cite[Thm. 4.1]{Pa94},  
the image of the restriction map, whose elements are called rational bundles, is the sublattice with basis 
$$
\{i(\rho_w)\cdot g_w,\  {w\in W}\},$$
where $i(\rho_w)$ is the index of the Brauer class $\beta(\bar{\rho_w})$, that is the index of any corresponding Tits algebra (see~\ref{Titsalgebra}).
\end{ntt}

\begin{ntt}\label{indcond}
Note that since the Weyl group acts trivially on $\Lambda_\omega/\T$, we have
$$
\bar\rho_w=\sum_{\{\alpha_k\in\Pi\mid w^{-1}(\alpha_k)\in\Phi^-\}} \bar\omega_k.
$$
Therefore, the corresponding Brauer class is given by 
$$
\beta(\bar\rho_w)=\Sigma_{\{\alpha_k\in\Pi\mid w^{-1}(\alpha_k)\in\Phi^-\}} 
\beta(\bar\omega_k).
$$
For special elements, we get $\beta(\bar \rho_{s_i})=\beta(\bar\omega_i)$, so that 
$i(\rho_{s_i})$ is the index of the Tits algebra $A_ {\omega_i}$. 
\end{ntt}

\subsection*{Rational cycles versus rational bundles}
Since the total Chern class of a rational bundle is a rational cycle, 
the graded-subring $\mathfrak{B}^*$ of $\CH^*(\BX)$
generated by Chern classes of rational bundles consists of rational cycles. 
We use Panin's description of rational bundles to compute $\mathfrak{B}^*$. 
The total Chern class of $i(\rho_w)\cdot g_w$ is given by 
$$
c(i(\rho_w)\cdot g_w)=
\bigl(1+c_1(\LL(\rho_w))\bigr)^{i(\rho_w)}=\sum_{k=1}^{i(\rho_w)}\binom{i(\rho_w)}{k}c_1(\LL(\rho_w))^k
$$
Therefore, $\mathfrak{B}^*$ is generated as a subring by the homogeneous elements 
$$
\binom{i(\rho_w)}{k} c_1(\LL(\rho_w))^k,\mbox{ for } w\in W,\ 1\leq k\leq i(\rho_w).
$$

Let $p$ be a prime number, and denote by $\ii_w$ the $p$-adic valuation of $i(\rho_w)$, so that 
$i(\rho_w)=p^{\ii_w}q$ for some prime-to-$p$ integer $q$. 
By Luca's theorem~\cite[p. 271]{Dick} the binomial coefficient $\binom{i(\rho_w)}{p^{\ii_w}}$ is congruent to $q$ modulo $p$. 
Hence its image in $\Fp$ is invertible. Considering the image in the Chow group modulo $p$ of the rational cycle 
$\binom{i(\rho_w)}{p^{\ii_w}} c_1(\LL(\rho_w))^{p^{\ii_w}}$, we get: 
 
\begin{lem}\label{ratcycl}
Let $p$ be a prime number. 
For any $w$ in the Weyl group, the cycle 
$$
c_1(\LL(\rho_w))^{p^{\ii_w}}\in \Ch(\BX)=\CH(\BX)\otimes_\zz\Fp
$$
is rational. 
\end{lem}

\section{The $J$-invariant}\label{JinvTits}

In this section, we recall briefly 
the definition and the key properties 
of the $J$-invariant of an algebraic group, following
\cite{PSZ}. The definition involves the choice of an ordering for a set of generators 
of $\Ch^*(G_0)$. For adjoint groups of type $\D_4$, as opposed to all other groups, there is no 
natural ordering. So the value of the $J$-invariant depends on a choice, 
which amounts, as we shall explain, to the choice of an underlying algebra with involution (see~\ref{JAI} below).  

\begin{ntt}
Let us denote by 
$\pi\colon\CH^*(\BX)\ra\CH^*(G_0)$ the pull-back induced by the natural projection $G_0\ra \BX$, where $\BX$ is the variety of Borel subgroups of $G_0$. 
By~\cite[\S 4, Rem. 2]{Gr58}, $\pi$ is surjective and its kernel is the ideal $I(\cc)\subset\CH^*(\BX)$ generated by 
non constant elements in the image of the characteristic map (see~\ref{defchar}). 
Therefore, there is an isomorphism of graded rings
\[
 \CH^*(\BX)/I(\cc)\simeq\CH^*(G_0)
\]
In particular, in degree $1$, we get 
\begin{equation}\label{chow}
\CH^1(G_0)\simeq \CH^1(\BX)/(\im\cc^{(1)})\simeq\Lambda_\omega/\T.
\end{equation}

Using this fact, Victor Kac computed in~\cite[Thm.~3]{Kc85} the Chow ring of $G_0$ with $\mathbb{F}_p$-coefficients. 
Namely, he proved it is isomorphic as an $\mathbb{F}_p$-algebra (and even as a Hopf algebra) to 
$$
\Ch^*(G_0)\simeq \Fp[x_1,\ldots,x_r]/(x_1^{p^{k_1}},\dots,x_r^{p^{k_r}})
$$
for some integers $r$ and $k_i$, for $1\leq i\leq r$, which depend on the group $G_0$.   
For each $i$, we let $d_i$ be the degree of the generator $x_i$. We always order the generators
so that $d_1\le\ldots\le d_r$.
Note in particular that the number of generators of degree $1$ is the dimension over $\Fp$ of the vector space 
$\Lambda_\omega/\T\otimes_\zz\Fp$. 

Let $G={}_\xi G_0$ for some $\xi\in Z^1(F,G_0)$.
The $J$-invariant of $G$ is related to the Chow-motif $M(\XX)$ of the corresponding twisted variety $\XX={}_\xi \BX$. 
Namely, the main result (Thm.~5.13) in~\cite{PSZ} asserts that the motive $M(\XX)$ 
splits as a direct sum
of twisted copies of some indecomposable motive $R_p(G)$, and
the Poincar\'e polynomial of $R_p(G)$ over a separable closure of $k$ (see \cite[\S1.3]{PSZ})
is given by 
\begin{equation}
\label{poincare}
P(R_p(G)\times_k k_s,t)=\prod_{i=1}^r \frac{1-t^{d_ip^{j_i}}}{1-t^{d_i}},\mbox{ where}\ 0\le j_i\le k_i.
\end{equation}
As opposed to 
$r$, $d_i$ and $k_i$ for $i\in 1\ldots r$, which only depend on $G_0$, the 
parameters $j_i$ depend on $G$ and reflect its motivic properties. 

Since the variety $\XX$ is uniquely determined by the group $G$ (see e.g.~\cite[\S1]{MPW}), it follows from this observation that the set $\{j_1,\dots,j_r\}$ is a well-defined invariant of the group $G$, i.e. it does not depend on any choice. 
Additionally, one may check from the values given in the table~\cite[Table II]{Kc85} (see also~\cite[\S 4]{PSZ}), 
that, except if $p=2$ and $G$ is adjoint of type $\D_n$ with $n$ even, the degrees $d_i$ are pairwise distinct. 
If so, the increasing hypothesis $d_1<\dots<d_r$ determines a canonical numbering of the generators, and we get a well defined $r$-tuple $J_p(G)=(j_1,j_2,\dots,j_r)$, which is the $J$-invariant of the group $G$. 
As opposed to this, $\Ch(\PGO_{4m})$ has two degree $1$ generators. Hence, to define the $J$-invariant for adjoint groups of type $\D_{2m}$, 
we need to order the set $\{j_1,j_2\}$, as we shall now explain. 
\end{ntt}

\begin{ntt} 
\label{computej1}
We recall from~\cite[Def. 4.6]{PSZ} the computation of the indices $j_i$ corresponding to the degree $1$ generator(s). 
Let $R$ be the pull-back in $\Ch^*(G_0)$ of the rational cycles of $\Ch^*(\BX)$, that is the image of the composition 
\begin{equation}\label{eqq}
R={\mathrm {Im}}\bigl(\Ch^*(\XX)\xrightarrow{\resch} \Ch^*(\XX_0)\xrightarrow{\pi} \Ch^*(G_0)\bigr),
\end{equation}
where the restriction map is as defined in~\ref{charinres}, and $\pi$ denotes as above the pull-back with respect to the projection $G_0\to \XX_0$. 

Let us denote by $s$ the dimension of $\Ch^1(G_0)$ over $\Fp$, and
assume first that $s=1$. We pick a generator $x_1\in \Ch^1(G_0)\simeq \Fp$. 
Then the parameter $j_1$ is the smallest non negative integer $a$ such that 
$x_1^{p^a}$ belongs to $R$ (see~\cite[\S 4.4]{PSZ}). 
Observe in particular that $j_1=0$ if and only if $x_1\in R$. 
\end{ntt}

\begin{ntt}\label{computej2}
We now assume that $p=2$ and $G_0$ is adjoint of type $\D_{2m}$, with $m\geq 2$, that is $G_0=\PGO_{4m}$. 
As mentioned earlier, this is the only case when $s>1$. 
Moreover, in view of~(\ref{chow}), choosing two degree $1$ generators for $\Ch^*(G_0)$ amounts to the choice of two generators of the cocenter $(\Lambda_\omega/\Lambda_r)=\zz/2\zz\times\zz/2\zz$ of the group. 
We pick the classes of the fundamental weights $\omega_1$ and $\omega_{2m}$ (see~\ref{cocenter}), and denote by $x_1$ and $x_2$ the corresponding elements in $\Ch^1(G_0)$,
\begin{equation} 
\label{x12}
  x_1=\pi(h_1),\ \ x_2=\pi(h_{2m}),
\text{ where }h_i=c_1(\LL(\omega_i)),\text{ as in \S\ref{char.section}.}
\end{equation}
Note that this choice is compatible with~\cite{PSZ}, since the relations $x_1^{2^{k_1}}=x_2^{2^{k_2}}=0$ are fulfilled.

By definition~\cite[\S 4.4]{PSZ}, $j_1$ is the smallest non negative integer $a$ such that 
$x_1^{2^a}$ belongs to $R$, and 
$j_2$ is the smallest integer $b$ such that 
$$
x_2^{2^b}+\sum_{0<i\leq 2^b} a_ix_1^ix_2^{{2^b}-i}\in R\text{ for some }
a_i\in \mf.
$$ 
Since $\omega_1+\omega_{n-1}+\omega_n\in \Lambda_r=\T$ (cf.~\ref{cocenter}), replacing $\omega_{2m}$ by $\omega_{2m-1}$ 
amounts to computing $(x_1+x_2)^{2^b}$ instead of $x_2^{2^b}$, and does not affect the value of $j_2$. 
Hence, the pair $(j_1,j_2)$ is well defined, once $x_1$ is chosen. 
\end{ntt}

\begin{ntt}
\label{JAI}
If we are not in the trialitarian case, that is if $m\not=2$, then the generator $x_1$ is canonically defined. 
Indeed, since we numbered roots as in~\cite{Bou}, 
the element $x_1$, which corresponds to the class 
of $\omega_1$ in $(\Lambda_\omega/\Lambda_r)\otimes\mf\simeq\Ch^1(G_0)$ is the unique element of $\Ch^1(G_0)$ killed by the pull-back map 
\[\Ch^1(G_0)\rightarrow \Ch^1(\SO_{4m})\] induced by taking the quotient map modulo the center. 
Hence, this provides an ordering for $\{j_1,j_2\}$, so that the $J$-invariant as an $r$-tuple is well defined in this case. 

As opposed to this, the $J$-invariant of a group of type $\D_4$, as we just defined it, 
does depend on a choice. Indeed, the three non-simply connected covers of $\PGO_8$, that is the special orthogonal group and the 
two half-spin groups are isomorphic. Hence $\omega_1$ and $x_1$ are not canonically defined. 
As a reflect of this fact (see~\cite[\S\,42]{BI}), the twisted group $G={}_\xi G_0$ can be described as the connected component 
of the automorphism group of three possibly non isomorphic degree $8$ algebras with orthogonal involution, which are the Tits algebras 
of the weights $\omega_1$, $\omega_3$ and $\omega_4$, 
$$G=\PGO(A,\sigma_A)=\PGO(B,\sigma_B)=\PGO(C,\sigma_C).$$
So, we use the notation $J(A,\sigma)$ for the triple $(j_1,j_2,j_3)$, where  
$j_1$ and $j_2$ are computed as in~\ref{computej2}, and $A$ is the Tits algebra of the weight $\omega_1$. 
Using a different choice for $\omega_1$, one can also compute $J(B,\sigma_B)$ and $J(C,\sigma_C)$. 

\end{ntt}
\begin{ntt}
\label{triality.rem}
Note that since the group is the same for all three algebras with involution, the variety $\XX$ and its motive do not depend on this choice. 
Hence, it follows from the formula for the Poincar\'e polynomial~\eqref{poincare} that 
$$J(B,\sigma_B),J(C,\sigma_C)\in \{(j_1,j_2,j_3),(j_2,j_1,j_3)\}.$$ 
In~\ref{trial.thm} and~\ref{12.ex} below, we give a more precise statement, and provide explicit examples of algebras with involution 
having isomorphic automorphism groups and different $J$-invariant. 
\end{ntt}

\begin{ntt}\label{extension} 
For further use, we finally recall that the $J$-invariant can only decrease under scalar extension, as explained in~\cite[4.7]{PSZ}, that is
\[
 J_p(G_E)\leq J_p(G),\mbox{ for any extension }E\text{ of the base field } F.
\]
\end{ntt}


\section{The parameters of degree one and indices of Tits algebras}
\label{results.sec}

In this section, we prove the main results of the paper, which give connections between the indices of the $J$-invariant corresponding to 
generators of degree $1$ and indices of Tits algebras of the group $G$ (cf. \cite[\S4]{PS09}). 

\begin{ntt}\label{il}
From now on, we let $s$ be the dimension over $\Fp$ of $\Lambda_\omega/\T\otimes \Fp\simeq\Ch^1(G_0)$, and we fix $G_0$ and a prime $p$ so that $s\geq 1$. 
We fix an index $i_1\in\{1,\dots,n\}$ if $s=1$ and two indices $i_1$ and $i_2$ if $s=2$, 
such that the class of $\omega_{i_1}$ (respectively the classes of $\omega_{i_1}$ and $\omega_{i_2}$) 
generate $\Lambda_\omega/\T\otimes\Fp$. 
We let $x_1=\pi(h_{i_1})\in\Ch^1(G_0)$ (and respectively $x_1=\pi(h_{i_1})$ and $x_2=\pi(h_{i_2})\in\Ch^1(G_0)$) be the corresponding generators. 

\end{ntt}
\begin{ntt}
Consider the special elements $g_i$, $i=1\ldots n$ 
of the Steinberg basis of $K_0(\BX)$ (see Definition~\ref{defspecel}). 
Since $c_1(g_i)=h_i-c_1(\LL(\alpha_i)) \in \Ch^1(\BX)$, we have 
$$
\pi(c_1(g_i))=\pi(h_i)-\pi(c_1(\LL(\alpha_i)))=\pi(h_i) \in \Ch^1(G_0).
$$
Hence the generators $x_{\ell}$, $1\leq\ell\leq s$ may also be defined 
by $x_{\ell}=\pi(c_1(g_{i_\ell}))$. 
In view of the isomorphism~(\ref{chow}), it follows that for any $g\in\Pic(\BX)$ its Chern class 
modulo $p$ can be written as 
\begin{equation}
\label{decomp}
c_1(g)=\sum_{\ell=1}^s a_\ell c_1(g_{i_\ell})\mod \im\cc^{(1)}\in \Ch^1(\BX)
\end{equation}
\end{ntt}

As an immediate consequence of rationality of cycles introduced 
in Lemma~\ref{ratcycl} we obtain a different proof of
the first part of \cite[Prop.~4.2]{PS09}: 

\begin{cor}\label{mainineq} 
The first entry $j_1$ of the $J$-invariant is bounded
$$j_1\leq  \ii_{i_1},$$
by the $p$-adic
valuation $\ii_{i_1}$ of the index of the Tits algebra $A_{\omega_{i_1}}$ associated to $\omega_{i_1}$.  
\end{cor}
\begin{proof}
We apply lemma~\ref{ratcycl} to the weight $\rho_{s_{i_1}}=\omega_{i_1}-\alpha_{i_1}$. 
As noticed in~\ref{indcond}, the index $i(\rho_{s_{i_1}})$ is equal to the index 
$i(\omega_{i_1})$ of the Tits algebra $A_{\omega_{i_1}}$. 
Hence, the cycle 
$c_1(g_{i_1})^{p^{\ii_{i_1}}}$ is rational, and its image ${x_{1}}^{p^{\ii_{i_1}}}\in\Ch^*(G_0)$ belongs to $R$.
The inequality then follows from the definition of $j_1$ (see \ref{computej1}). 
\end{proof}

Assume now that $s=2$, so that $p=2$ and $G_0$ is adjoint of type $\D_{2m}$. 
As explained in~\ref{computej2}, we may take $i_1=1$, $i_2=2m$, and define $x_1$ and $x_2$ as in~(\ref{x12}). 
If moreover $m=2$, that is $G_0$ is adjoint of type $\D_4$, the $J$-invariant refers 
to the $J$-invariant of $(A,\sigma)$, where $A$ is the Tits algebra associated to the weight $\omega_1$. 
The computation in the proof above also applies to $\rho_{s_{\ell}}$ for $\ell=2m-1$ and $\ell=2m$. 
Combining the result with the definition 
of $j_2$ given in~\ref{computej2}, we get
\begin{cor} 
\label{mainineq2}
If $p=2$ and $G$ is adjoint of type $\D_{2m}$, so that $s=2$, then 
the second entry $j_2$ of the $J$-invariant is bounded 
$$j_2\leq \min\{\ii_{2m-1},\ii_{2m}\},$$
where $\ii_\ell$ is the $2$-adic valuation of the index of the Tits algebra $A_{\omega_{\ell}}$. 
\end{cor}
 
The next result, which gives an inequality in the other direction, uses the notion of common index, which we introduce now. 
\begin{dfn} Consider the Tits algebras $A_{\omega_{i_\ell}}$ associated to the fundamental weights $\omega_{i_\ell}$, for $1\leq\ell\leq s$, where $i_\ell$ are as in~\ref{il}. 
We define their {\em common index} $\ii_J$ to be the $p$-adic valuation
of the greatest common divisor of all the indices
$\ind(A_{\omega_{i_1}}^{\otimes a_1}\otimes\ldots\otimes A_{\omega_{i_s}}^{\otimes a_s})$,
where at least one of the $a_i$ is coprime to $p$. 
\end{dfn}

\begin{ex}
\label{excommonindex1} 
If $s=1$, then $\ii_J$ is the $p$-adic valuation $\ii_{i_1}$ of the index of the Tits algebra $A_{\omega_{i_1}}$. 
Assume for instance that $G$ is adjoint of type $\D_{2m+1}$. 
As recalled in~\ref{cocenter}, we may take $i_1=2m$ or $i_1=2m+1$, so that $\ii_J$ is the $2$-adic valuation of any component 
$\C_+$ or $\C_-$ of the Clifford algebra of $(A,\sigma)$. From the fundamental relations~\ref{FR}, we know that the two components have the same index.  
\end{ex}
\begin{ex}
\label{excommonindex2}
If $s=2$, we have $p=2$ and $G_0$ is adjoint of type $\D_{2m}$. 
Using~\ref{FR}, one may check that $\ii_J$ is the $p$-adic valuation of the 
greatest common divisor of the indices of $A_{\omega_1}$, $A_{\omega_{2m-1}}$ and $A_{\omega_{2m}}$, 
that is 
\[\ii_J=\mbox{min}\{\ii_1,\ii_{2m-1},\ii_{2m}\}.\]
\end{ex}

We will prove: 
\begin{thm}\label{corj0}
Let $\ii_J$ be the common index of the Tits algebras $A_{\omega_{i_\ell}}$, for $1\leq\ell\leq s$. 

If $\ii_J>0$, then $j_\ell> 0$ for any $\ell$, $1\leq\ell\leq s$.

If $\ii_J>1$ and $p=2$, then for any $\ell$ such that $k_\ell>1$, the corresponding index also satisfies 
$j_\ell>1$.
\end{thm}

Consider the ideal $I(\resch)$ of $\Ch^*(\BX)$ generated by the non constant rational elements. 
For any integer $i$, we let $I(\resch)^{(i)}\subset \Ch^i(\BX)$ be the homogeneous part of degree $i$. 
Since the image of the characteristic map consists of rational elements, we have 
$I(\cc)\subset I(\resch)$. The theorem follows immediately from the following lemma: 
\begin{lem}
\label{ideals}
 If $\ii_J>0$, then $I(\resch)^{(1)}=I(\cc)^{(1)}\subset\Ch^1(\BX)$. 

If $\ii_J>1$ and $p=2$, then $I(\resch)^{(2)}=I(\cc)^{(2)}\subset\Ch^2(\BX)$
\end{lem}

Indeed, let us assume first that $\ii_J>0$. 
By the lemma, any element in $\im(\resch^{(1)})=I(\resch)^{(1)}$ belongs to 
$I(\cc)^{(1)}$, which is in the kernel of $\pi$. 
Therefore, the image of the composition 
\[R^{(1)}=\im\bigl(\Ch^1(\XX)\xrightarrow{\resch^{(1)}}  \Ch^1(\XX_0)\xrightarrow{\pi} \Ch^1(G_0)\bigr)\]
is trivial, $R^{(1)}=\{0\}$. 
From the definition~\ref{computej2} of $j_1$ and $j_2$, this implies that they are both strictly positive.  

The proof of the second part follows the same lines. We write it in details for $s=2$ and 
$k_1,k_2>1$. Assume that $\ii_J>1$. 
Since the image $\im(\resch)^{(2)}$ is contained in $I(\resch)^{(2)}$, the lemma again implies that 
$R^{(2)}=\{0\}$. 
On the other hand, the hypothesis on $k_1$ and $k_2$ guarantees that 
in the truncated polynomial algebra $\mf[x_1,x_2]/(x_1^{2^{k_1}},x_2^{2^{k_2}})\subset \Ch^*(G_0)$, the elements 
$x_1^2$ and $x_2^2+a_1x_1x_2+a_2x_1^2$ are all non trivial. 
Hence they do not belong to $R$, and we get $j_1,j_2>1$. 

The rest of the section is devoted to the proof of Lemma~\ref{ideals}. 
The main tool is the Riemann-Roch theorem, which we now recall.

\subsection*{Filtrations of $K_0$ and the Riemann-Roch Theorem}

\begin{ntt}
Let $X$ be a smooth projective variety over $F$. 
Consider the topological filtration on $K_0(X)$ given by 
\[\tau^iK_0(X)=\langle[{\mathcal O}_V],\ \mbox{codim} V\geq i\rangle,\]
where ${\mathcal O}_V$ is the structure sheaf of the closed subvariety $V$ in $X$. 
There is an obvious surjection 
\[p\colon\CH^i(X)\ra \tau^{i/i+1}K_0(X)=\tau^iK_0(X)/\tau^{i+1}K_0(X),\]
given by $V\mapsto[{\mathcal O}_V]$. 
By the Riemann-Roch theorem without denominators~\cite[\S 15]{Fu}, the $i$-th Chern class induces a map in the opposite direction
\[c_i\colon\tau^{i/i+1}K_0(X)\ra \CH^i(X)\] and the composite $c_i\circ p$ is the multiplication by $(-1)^{i-1}(i-1)!$. 
In particular, it is an isomorphism for $i\leq 2$ (see~\cite[Ex. 15.3.6 ]{Fu}). 

The topological filtration can be approximated by the so-called $\gamma$-filtration. 
Let $c_i^{K_0}$ be the $i$-th Chern class with values in $K_0$ (see~\cite[Ex. 3.2.7(b)]{Fu}, 
or~\cite[\S 2]{Ka98}). We use the convention $c_1^{K_0}([\LL])=1-[{\LL^{v}}]$ for any line bundle $\LL$, where $\LL^{v}$ is the dual of $\LL$, so that in the Chow group, \[c_1(c_1^{K_0}([\LL]))=c_1(\LL).\] 
Similarly, one may compute the second Chern class 
\begin{equation} 
\label{c2}
c_2\bigl(c_1^{K_0}([\LL_1])c_1^{K_0}([\LL_2])\bigr)=-c_1(\LL_1)c_1(\LL_2).
\end{equation}

The $\gamma$-filtration on 
$K_0(X)$ is given by the
subgroups (cf. \cite[\S1]{GZ10})
$$
\gamma^iK_0(X)=
\langle c_{n_1}^{K_0}(b_1)\cdot \ldots \cdot c_{n_m}^{K_0}(b_m) \mid
n_1+\ldots + n_m\ge i,\; b_l\in K_0(X)\rangle,
$$
(see \cite[Ex.15.3.6]{Fu}, \cite[Ch.3,5]{FL}).  
We let 
$\gamma^{i/i+1}(K_0(X))=\gamma^iK_0(X)/\gamma^{i+1}K_0(X)$ 
be the respective quotients, and
$\gamma^*(X)=\oplus_{i\ge 0}\gamma^{i/i+1}(K_0(X))$ the
associated graded ring. 

By~\cite[Prop.~2.14]{Ka98}, $\gamma^i(K_0(X))$ is contained in $\tau^i(K_0(X))$, and they coincide for $i\leq 2$. 
Hence, by the Riemann-Roch theorem, the Chern class $c_i$ with values in $\CH^i(X)$ vanishes on $\gamma^{(i+1)}K_0(X)$,  
and induces a map 
\[c_i\colon\gamma^{i/i+1}(K_0(X))\ra \CH^i(X).\] 
In codimension $1$ we get an isomorphism
$$
c_1\colon \gamma^{1/2}(K_0(X))\xrightarrow{\simeq}\mathrm{CH}^1(X)
$$
which sends for a line bundle $\LL$ the class $c_1^{K_0}(\LL)$ to $c_1(\LL)$.
In codimension $2$ the map
$$
c_2\colon \gamma^{2/3}(K_0(X)) \twoheadrightarrow \mathrm{CH}^2(X),
$$
is surjective and has torsion kernel~\cite[Cor. 2.15]{Ka98}.

Let us now apply this to the varieties $\BX$ and $\XX$ of Borel subgroups of $G_0$ and $G$ respectively. 
Since $K_0(\BX)$ is generated by the line bundles $g_w=[\LL(\rho_w)]$ for $w\in W$, one may check that 
$\gamma^{i/i+1}(\BX)$ is generated by the products \[\{c_1^{K_0}(g_{w_1})\dots c_1^{K_0}(g_{w_i}),\ w_1,\dots,w_i\in W\}.\] 
Moreover, the restriction map commutes with Chern classes, so it induces 
\[
 \resg\colon\gamma^*(\XX)\ra\gamma^*(\BX).
\]
Using Panin's description of the image of the restriction map $\resk$ recalled in~\ref{Panin}, we get that 
the image of $\resg^{(1)}\colon\gamma^{1/2}(\XX)\ra\gamma^{1/2}(\BX) $ is generated by 
the elements $c_1^{K_0}(i(\rho_w)g_w)=i(\rho_w)c_1^{K_0}(g_w)$, for any $w\in W$, while  
the image of $\resg^{(2)}$ is generated by 
\[i(\rho_{w_1})i(\rho_{w_2})c_1^{K_0}(g_{w_1})c_1^{K_0}(g_{w_2})\mbox{ and } 
c_2^{K_0}(i(\rho_w)g_{w})\mbox{ for }w_1,w_2,w\in W.\] 
If the index $i(\rho_w)$ is $1$, then $c_2^{K_0}(i(\rho_w)g_{w})=0$. 
Otherwise, the Whitney sum formula gives 
\[
 c_2^{K_0}(i(\rho_w)g_w)=\binom{i(\rho_w)}{2}c_1^{K_0}(g_w)^2.
\]
Applying the morphisms $c_1$ and $c_2$, and using~(\ref{c2}), we now get 
\begin{lem}  \label{ci}
 The subgroup $c_1\bigl(\im(\resg^{(1)})\bigr)\in\CH^1(\BX)$ is generated by 
$i(\rho_w)c_1(g_w)$, for all $w\in W$. 
The subgroup $c_2\bigl(\im(\resg^{(2)})\bigr)\in\CH^2(\BX)$ is generated by the elements 
$i(\rho_{w_1})i(\rho_{w_2})c_1(g_{w_1})c_1(g_{w_2})$ and $\binom{i(\rho_w)}{2}c_1(g_w)^2$ for all $w_1,w_2,w\in W$. 
\end{lem}
\end{ntt}

\subsection*{Proof of Lemma~\ref{ideals}}

Since the image of the characteristic map consists of rational elements (see~\ref{charinres}), we already know that 
$I(\cc)\subset I(\resch)$. We now prove the reverse inclusions for the homogeneous parts of degree $1$ and $2$ under the relevant hypothesis on the common index $\ii_J$. 
Note that since $c_1$ and $c_2$ are both surjective, and commute with restriction maps, one has
 \[\im\bigl(\resch^{(k)}\bigr)=c_k\bigl(\im(\resg^{(k)})\bigr),\mbox{ for }k=1,2.\]
In degree $1$, we have $I(\resch)^{(1)}=\im(\resch^{(1)})$, so to prove the first part of the lemma, we have to prove that 
if $\ii_J>0$, then for any $w\in W$, the element $i(\rho_w)c_1(g_w)$ belongs, after tensoring with $\Fp$, to $I(\cc)^{(1)}=\im\cc^{(1)}$. 
Let us write \[c_1(g_w)=\sum_{\ell=1}^s a_\ell c_1(g_{i_\ell})\mod \im\cc^{(1)},\] as in~(\ref{decomp}). 
If all the $a_\ell\in\Fp$ are trivial, we are done, so we may assume at least one of them is invertible in $\Fp$. 
The weights $\rho_w$ and $\rho_{i_\ell}$ satisfy the same relation 
\[\rho_w=\sum_{\ell=1}^s a_\ell\rho_{i_\ell}\mod \T\otimes_\zz\Fp.\] 
Applying the morphism $\beta$, we get that the $p$-primary part of the Brauer class $\beta(\bar\rho_w)$ 
coincides with the $p$-primary part of the Brauer class of $\otimes_{\ell=1}^sA_{\omega_{i_\ell}}^{a_\ell}$ (see~\ref{indcond}). 
The hypothesis on $\ii_J$ guarantees that this index of this algebra is divisible by $p$. 
Hence $i(\rho_w)$, which is the index of $\beta(\bar\rho_w)$ also is divisible by $p$, so that 
$i(\rho_w)c_1(g_w)=0$ in the Chow group $\Ch^1(\BX)$ modulo $p$, and we are done. 

Let us now assume that $p=2$ and $\ii_J>1$. 
The homogeneous part $I(\resch)^{(2)}$ decomposes as 
\[I(\resch)^{(2)}=\im\bigl(\resch^{(1)}\bigr)\Ch^1(\BX)+ \im\bigl(\resch^{(2)}\bigr).\]
By the first part of the Lemma, we already know that 
\[\im\bigl(\resch^{(1)}\bigr)\Ch^1(\BX)\subset I(\cc).\]
Hence it remains to prove that $\im(\resch^{(2)})=c_2(\im\resg^{(2)})\subset I(\cc)^{(2)}$. 
The proof for the degree $1$ part already shows that $i(\rho_{w_1})i(\rho_{w_2})c_1(g_{w_1})c_1(g_{w_2})$ belongs 
to $I(\cc)^{(2)}$. The same argument extends to $\binom{i(\rho_w)}{2}c_1(g_w)^2$. 
Indeed, if the coefficients $a_\ell$ are not all trivial modulo $2$, the condition on the common index 
now implies that $4$ divides $i(\rho_w)$, so that $\binom{i(\rho_w)}{2}$ is zero modulo $2$. 


\section{Applications to quadratic forms}\label{applquad}

The rest of the paper is now devoted to explicit computations of the $J$-invariant in some low dimensional examples.
Mostly, we are interested in the trialitarian case, that is we want to compute the $J$-invariant of degree $8$ algebras
with orthogonal involution (see Thm.~\ref{trial.thm} below). As explained in Cor.~\ref{values.cor}, this computation has some interesting consequences on the set of possible values for the $J$-invariant. 

The starting point for this computation is the quadratic form case. 
In small dimension, the $J$-invariant can be computed in terms of classical invariants of quadratic forms such as the Schur index of its Clifford algebra and the Witt index. 
This was already noticed (see~\cite[\S 88]{EKM}), and can be recovered very easily using our inequalities~\ref{corj0}. 
Since the results are used in the next section, we do include precise statements and proofs; on the way, we also provide an explicit example showing that the splitting pattern is a finer invariant than the $J$-invariant in dimension $\geq 10$.

Let $\varphi$ be a quadratic form of even dimension $2n$. We always assume that $\varphi$ has trivial discriminant, so that its 
special orthogonal group $\SO(\varphi)$ satisfies condition~\ref{inner}. 
We define the $J$-invariant of $\varphi$ as follows: 
\begin{dfn}
Let $\varphi$ be a $2n$ dimensional quadratic form over $F$, with trivial discriminant. 
Its $J$-invariant is 
$$J(\varphi)=J_2\bigl(\SO(\varphi)\bigr).$$
\end{dfn}
\begin{rem}
(i)The $J$-invariant of a quadratic form was initially defined by Vishik in~\cite[Def 5.11]{Vi05}. A dual version 
is also given in~\cite[\S 88]{EKM}. Note that these are three different versions of the same invariant, meaning that any of them determines the other two. We refer the reader to \cite[\S 4.8]{PSZ} for a more precise statement, and to the table below for an example of the dictionary. 

(ii)Let $\varphi_0$ be any non degenerate subform of $\varphi$ of codimension $1$. Since $\varphi$ has trivial discriminant, $\varphi$ and $\varphi_0$ have the same splitting fields. In particular, each of them splits over the function field of the maximal orthogonal grassmanian of the other. 
Therefore by the comparison lemma~\cite[5.18(iii)]{PSZ}, the corresponding indecomposable motives $R_2\bigl(\SO(\varphi_0)\bigr)$ and 
$R_2\bigl(\SO(\varphi)\bigr)$ are isomorphic. 
Hence they have the same Poincar\'e polynomial, and by~(\ref{poincare}), it follows that $\SO(\varphi_0)$ and $\SO(\varphi)$ have the same $J$-invariant. 
Since any odd-dimensional form can be embedded in an even dimensional form with trivial discriminant, we only consider the even-dimensional case in the sequel. 
\end{rem}

The argument given in remark (ii) above also applies to Witt equivalent quadratic forms, proving the first assertion of the following proposition(see also~\cite[\S 88]{EKM}): 
\begin{prop}
\label{witt.prop}
Let $\varphi$ and $\varphi'$ be two Witt-equivalent even-dimensional quadratic forms with trivial discriminant. 
All the non trivial indices in their $J$-invariant are equal. 

Moreover, $J(\varphi)=(0,\ldots,0)$ if and only if $\varphi$ is hyperbolic. 
\end{prop}
\begin{proof}
The second assertion follows from Springer's theorem~\cite[18.5]{EKM}. 
Indeed, as explained in~\cite[6.7]{PSZ}, a quadratic form with trivial $J$-invariant is hyperbolic over some odd-degree extension of the base field. 
\end{proof}

In small dimension, the $J$-invariant of a quadratic form $\varphi$ can be explicitly computed in terms of the index of its full Clifford algebra $\C(\varphi)$. 
More precisely, we let $\ii_S$ be the $2$-adic valuation of the index of $\C(\varphi)$. 
Since $\varphi$ has trivial discriminant, its even Clifford algebra has center $F\times F$ and splits as 
$\C_0(\varphi)=C\times C$ for some central simple algebra $C$ which is Brauer-equivalent to $\C(\varphi)$, that is $\C(\varphi)\simeq M_2(C)$. 
In particular, it follows that $\ii_S\leq n-1$. 
In this setting, we have $s=1$, and $C$ is the Tits algebra associated to a generator of $\Lambda_\omega/\T\otimes \mf$. 
So the common index $\ii_J$ is given by $\ii_J=\ii_S$. 
With this notation, the inequalities~\ref{mainineq} and~\ref{corj0} can be translated as follows: 
\begin{cor}
Let $\varphi$ be a $2n$ dimensional quadratic form with trivial discriminant. 
The $2$-adic valuation $\ii_S$ of its Clifford algebra and the first index $j_1$ of its $J$-invariant are related as follows: 
\begin{enumerate}
\label{qf.cor}
\item $j_1\leq \ii_S$.  
\item If $n\geq 2$, and $\ii_S>0$, then $j_1>0$.
\item If $n\geq 3$ and $\ii_S>1$, then $j_1>1$.
\end{enumerate} 
\end{cor}

Assume now that $\varphi$ has dimension $4$ or $6$, that is $n=2$ or $3$.
By table~\cite[4.13]{PSZ}, the $J$-invariant of $\varphi$ consists of a single integer, 
$J(\varphi)=(j_1)$, which is bounded by $1$ (respectively $2$) if $\varphi$ has dimension $4$ (respectively $6$). 
Therefore, by Corollary~\ref{qf.cor}, we have:  
\begin{prop}
Let $\varphi$ be a quadratic form of dimension $4$ or $6$ with trivial discriminant. 
Its $J$-invariant is $J(\varphi)=(\ii_S)$. 
\end{prop}
We can give a more precise description of the quadratic form $\varphi$ in each case. 
In dimension $4$, the Clifford algebra of $\varphi$ is Brauer-equivalent to a quaternion algebra $Q$ over $F$, and $\varphi$ is 
similar to the norm form $n_Q$ of $Q$, which is a $2$-fold Pfister form. It is hyperbolic if $Q$ is split and anisotropic otherwise. 

If $\varphi$ has dimension $6$, its Clifford algebra is a biquaternion algebra $B$ over $F$, and 
$\varphi$ is an Albert form of $B$. 
If $J(\varphi)=(2)$, or equivalently $B$ is division, then $\varphi$ is anisotropic. 
If $J(\varphi)=(1)$, or equivalently $B$ is Brauer equivalent to a non split quaternion algebra $Q$, then 
$\varphi$ is similar to $n_Q\oplus{\mathbb H}$. If $J(\varphi)=(0)$, or equivalently $B$ is split, then $\varphi$ is hyperbolic. 
Hence we get:
\begin{cor}
\label{qf6}
A quadratic form $\varphi$ of dimension $6$ and trivial discriminant is anisotropic if and only if its $J$-invariant is 
$J(\varphi)=(2)$; it is isotropic and non-hyperbolic if and only if $J(\varphi)=(1)$. 
\end{cor}

Let us now consider quadratic forms of dimension $8$ with trivial discriminant. Their special orthogonal groups have type $\D_4$, and table~\cite[4.13]{PSZ} now says 
$J(\varphi)=(j_1,j_2)$ with $0\leq j_1\leq 2$ and $0\leq j_2\leq 1$. 
We have: 
\begin{prop}
\label{qf8}
Let $\varphi$ be a quadratic form of dimension $8$ with trivial discriminant, and consider its $J$-invariant $J(\varphi)=(j_1,j_2)$. 
The first index $j_1$ is given by $j_1=\min\{\ii_S,2\}$. 
Moreover, $j_2$ is $0$ if $\varphi$ is isotropic and $1$ if $\varphi$ is anisotropic. 
\end{prop}
\begin{proof}
Again, the first assertion follows instantly from Corollary~\ref{qf.cor}. 
To prove the second, let us first assume that $\varphi$ is isotropic. 
If it is hyperbolic, we already know that $J(\varphi)=(0,0)$, so in particular $j_2=0$. 
Otherwise, $\varphi$ splits as $\varphi=\varphi_0\oplus{\mathbb H}$ for some $6$-dimensional 
non hyperbolic form $\varphi_0$ with trivial discriminant. By~\ref{witt.prop}, this implies that one of the two indices $j_1$ and $j_2$ of $J(\varphi)$ is zero. 
Since the Clifford algebra of $\varphi$ is Brauer-equivalent to the Clifford algebra of $\varphi_0$, which is non split 
by~\ref{qf6}, $j_1=\ii_S\not=0$. Therefore, we get $j_2=0$ as expected.  

To prove the converse, we distinguish two cases, depending on the value of $\ii_S$. 
Let us first assume that $\varphi$ is anisotropic and $\ii_S\leq 2$. Consider a generic splitting $F_C$ 
of the Clifford algebra of $\varphi$. By a theorem of Laghribi~\cite[Thm.~4]{La96}, the form $\varphi$ remains
anisotropic after scalar extension to $F_C$. Hence the $J$-invariant of $\varphi_{F_C}$ is non-trivial. 
On the other hand, since its Clifford algebra is split, the first index is zero. Therefore the second index 
$j_2$ is $1$, and this is a fortiori the case over the base field. 

Finally, let us assume $\ii_S=3$, which implies in particular that $\varphi$ is anisotropic. 
Any field extension over which the quadratic form $\varphi$ becomes hyperbolic splits 
its Clifford algebra. Therefore the index of such a field extension has to be at least $8=2^3$. 
Hence, by~\cite[6.6]{PSZ}, we have $3\leq j_1+j_2$, so that $J(\varphi)=(2,1)$. 
\end{proof}

In his paper~\cite{Ho98}, Hoffmann classified quadratic forms of small dimension in terms of their splitting pattern. 
Combining his classification with the previous proposition, one can give a precise description of quadratic forms of dimension $8$ with trivial discriminant, depending on the value of their $J$-invariant. 
In particular, this description shows that all possible values for the $J$-invariant do occur in this setting. 

The results are summarized in the table below. 
The notation $J_v(\varphi)$ stands for Vishik's $J$-invariant, as defined in~\cite[\S 88]{EKM}. 
The index $\ii$ is the $2$-adic valuation of the greatest common divisor of the degrees of the splitting fields of $\varphi$. 
In the explicit description, $Pf_k$ stands for a $k$-fold Pfister form, $s_{l/k}(Pf_2)$ for the Scharlau transfer of a $2$-fold Pfister form with respect to a quadratic field extension, and $Al_6$ for an Albert form. 

\medbreak

\noindent
{
\small
\begin{tabular}{cc|c|c||c}
 $J(\varphi)$ & $J_v(\varphi)$ & $\ii_S$, $\ii$ & Splitting Pattern 
& Description\\
 (0) & $\emptyset$ & $\ii_S=\ii=0$ & (4) & hyperbolic\\
(1,0) & $\{1\}$  & $\ii_S=\ii=1$ &  (2,4) & $Pf_2\perp 2\mathbb{H}$\\
(2,0) & $\{1,2\}$ & $\ii_S=\ii=2$ & (1,2,4) & $Al_6 \perp \mathbb{H}$\\
(0,1) & $\{3\}$ & $\ii_S=0$; $\ii=1$ & (0,4) & $Pf_3$ \\
(1,1) & $\{1,3\}$ & $\ii_S=\ii=1$ & (0,2,4) & $q=\langle1,-a\rangle\otimes q'$ \\
 (2,1) &$\{1,2,3\}$ & $\ii_S=\ii=2$  & (0,1,2,4) & $Pf_2\perp Pf_2$ or $s_{l/k}(Pf_2)$ \\
            &       & $\ii_S=\ii=3$ & `` & generic
\end{tabular}
}

\medbreak

\begin{rem}
One may observe from the table
that the $J$-invariant and the splitting pattern uniquely determine each other for $8$ dimensional quadratic forms. 
This is not true anymore in higher degree. 

Indeed, consider a $10$-dimensional quadratic form $\varphi$ over $F$
with splitting pattern 
$(0,2,3,5)$. By Hoffmann's classification theorem~\cite[Thm~5.1]{Ho98}, $\varphi$ has trivial discriminant, its Clifford algebra has index $4$, so that $\ii_S=2$ and it is a Pfister
neighbor. We claim that, even though this form is anisotropic, its $J$-invariant is $J(\varphi)=(2,0)$. 
To prove this, one may use the relation between the splitting pattern and Vishik's $J$-invariant of a quadratic form as described in~\cite[88.8]{EKM}. Once translated in terms
of $J(\varphi)$ following~\cite[4.8]{PSZ}, we get that $J(\varphi)$ is $(j_1,0)$ for some integer $j_1$, $1\leq j_1\leq 3$. Since $\ii_S=2$, Corollary~\ref{qf.cor} give $j_1=2$. 

From the classification in degree $8$, the value $(2,0)$ also is the $J$-invariant of $\pi+2\mathbb{H}$ for any anisotropic Albert form $\pi$. 
On the other hand, such a form has splitting pattern $(2,3,5)$. 
\end{rem}


\section{$J$-invariant of an algebra with involution}\label{applalg}

In this section, we apply the inequalities proven in section~\ref{results.sec} to the case of an algebra with orthogonal involution. 
As an immediate consequence, we compute the $J$-invariant in degree $\leq 6$.

Recall that $(A,\sigma)$ is a degree $2n$ central simple algebra over $F$, endowed with an involution of orthogonal type and trivial discriminant. 
In particular, this implies that $A$ has exponent $2$, so that it has index $2^{\ii_A}$ for some integer $\ii_A$.  
The connected component $\PGO(A,\sigma)$ of the automorphism group of $(A,\sigma)$ is an adjoint group of type $\D_n$. 
Because of the discriminant hypothesis, it is an inner twisted form of $\PGO_{2n}$. 
If the degree of $A$ is different from $8$, we define 
\[J(A,\sigma)=J_2\bigl(\PGO(A,\sigma)\bigr).\]
In degree $8$, $J(A,\sigma)$ was defined in~\ref{JAI}. 
Therefore, from the table~\cite[4.13]{PSZ}, one may check that $J(A,\sigma)$ is an $r$-tuple $J(A,\sigma)=(j_1,j_2,\dots,j_r)$, with $r=m+1$ if $n=2m$ and $r=m$ if $n=2m+1$.  
Note that our notation slightly differs from the notation in the table, where in the $n$-odd case, they have an additional index, but which is bounded by $k_1=0$. 
So, for $n$ odd, our $(j_1,\dots,j_r)$ coincides with $(j_2,\dots,j_{r+1})$ in~\cite{PSZ}. 
In particular, the indices corresponding to generators of degree $1$ are $j_1$ if $n$ is odd and $j_1$ and $j_2$ if $n$ is even. 

Since $\sigma$ has trivial discriminant, its Clifford algebra splits as a direct product 
$\C(A,\sigma)=\C_+\times\C_-$ of two central simple algebras over $k$. 
We let $\ii_A$ (respectively $\ii_+$, $\ii_-$) be the $2$-adic valuation of the index of $A$ (respectively $\C_+$, $\C_-$). 
From Examples~\ref{excommonindex1} and~\ref{excommonindex2}, the common index $\ii_J$ is 
\[\ii_J=\left\{\begin{array}{ll} 
                \ii_+=\ii_-&\mbox{ if $n$ is odd,}\\
\min\{\ii_A,\ii_+,\ii_-\}&\mbox{ if $n$ is even.}
               \end{array}\right.\]
Hence, Corollaries~\ref{mainineq} and~\ref{mainineq2} and Theorem~\ref{corj0} translate as follows: 

\begin{cor}
\label{odd.cor}
Assume that $n$ is odd, so that $\deg(A)\equiv 2[4]$, and let $\ii_S=\ii_+=\ii_-$. 
We have: 
\begin{enumerate}
\item $j_1\leq \ii_S$;
\item  If $\ii_S>0$, then $j_1>0$;
\item If $\deg(A)\geq 6$ and $\ii_S>1$, then $j_1>1$.  
\end{enumerate}
\end{cor}
\begin{cor} 
\label{even.cor}
 Assume now that $n$ is even, that is $\deg(A)\equiv 0[4]$, and let 
$\ii_J=\min\{\ii_A,\ii_+,\ii_-\}$.  
We have: 
\begin{enumerate} 
 \item $j_1\leq\ii_A$; 
\item $j_2\leq\min\{\ii_+,\ii_-\}$;
\item If $\ii_J>0$, then $j_1>0$ and $j_2>0$. 
\item If $\deg(A)\equiv 0[8]$ and $\ii_J>1$, then $j_1>1$. 
\item If $\deg(A)\geq 8$ and $\ii_J>1$, then $j_2>1$. 
\end{enumerate}
\end{cor}
The additional conditions on the degrees are obtained from the table~\cite[4.13]{PSZ}, and guarantee that 
$k_1>1$ or $k_2>1$. 

\subsection*{Split case}
If $A$ is split, the involution $\sigma$ is adjoint to a quadratic form $\varphi$ over $k$. 
We then have: 
\begin{prop} 
\label{split.prop}
If $A$ is split and $\sigma$ is adjoint to the quadratic form $\varphi$, the $J$-invariants of $(A,\sigma)$ and $\varphi$ are related as follows: 
\[J(A,\sigma)=\left\{\begin{array}{ll} 
                      J(\varphi)&\mbox{ if }\deg(A)\equiv 2[4]\\
\bigl(0,J(\varphi)\bigr)&\mbox{ if }\deg(A)\equiv 0[4]\\
                     \end{array}\right.\]
\end{prop}
\begin{proof} 
Assume $A$ is split, and $\sigma$ is adjoint to the quadratic form $\varphi$. 
Since $d(\varphi)=d(\sigma)=1$, the $J$-invariant of $\varphi$ is defined, 
\[J(\varphi)=J_2\bigl(\SO(\varphi)\bigr).\]
Moreover, the groups $\SO(\varphi)=\SO(A,\sigma)$ and $\PGO(A,\sigma)=\PGO(\varphi)$ are isogeneous. 
Therefore, the corresponding varieties of Borel subgroups are the same, so that $R_2\bigl(\SO(\varphi)\bigr)= R_2\bigl(\PGO(\varphi)\bigr)$. 
Hence, by the formula~(\ref{poincare}) for the Poincar\'e polynomial of this indecomposable motive, 
the non trivial indices in the $J$-invariants of $\varphi$ and $(A,\sigma)$ are the same. 
If $n$ is odd, this is enough to conclude that the $J$-invariants are equal. If $n$ is even, the only difference comes from the presentations of 
$\Ch^*(\SO_{2n})$ and $\Ch^*(\PGO_{2n})$: the second group has two generators of degree $1$, while the first has only one. 
Since we precisely defined $x_1$ to be the generator of $\Ch^1(\PGO_{2n})$ killed by pull-back to $\Ch^1(\SO_{2n})$ (see \S\ref{JAI}), 
we get that the index $j_1$ of $J(A,\sigma)$ is trivial when $A$ is split, and this proves the proposition. 
\end{proof}

\subsection*{Half-spin case}
If the Clifford algebra $\C(A,\sigma)=\C_+\times\C_-$ has a split component, and if $\deg(A)\equiv 2 [4]$, by the fundamental relations~\ref{FR}, the algebra $A$ is split, so that 
$J(A,\sigma)=J(\varphi)$ for a suitable quadratic form $\varphi$. 

If $\deg(A)\equiv 0[4]$, we now prove: 
\begin{prop} 
\label{hs.prop}
Assume that $\deg(A)\equiv 0[4]$. 
The Clifford algebra $\C(A,\sigma)=\C_+\times\C_-$ has a split component if and only if one of the two half-spin groups, say 
$\Spin^+(A,\sigma)$, satisfies condition~\ref{inner}. 
If so, the algebra with involution $(A,\sigma)$ is said to be half-spin, and its $J$-invariant satisfies 
\[J(A,\sigma)=(j_1,0,j_3,\dots,j_r)\mbox{, where }(j_1,j_3,\dots,j_r)=J_2\bigl(\Spin^+(A,\sigma)\bigr).\]
\end{prop}

\begin{proof}
Let us pick one of the two (isomorphic) half-spin groups $\Spin^+_{2n}\subset \Spin_{2n}$, and consider the exact sequence
\[1\ra \mu_2\ra \Spin^+_{2n}\ra \PGO_{2n}\ra 1\]
As explained in \cite[\S29]{BI}, the isomorphism class of $(A,\sigma)$, which can be viewed as an element of $H^1(F,\text{PGO}_{2n})$, has two different liftings in $H^1(F,\PGO_{2n})$. Moreover, their images under the connecting morphism 
\[H^1(F,\PGO_{2n})\ra H^2(F,\mu_2),\]
are the Brauer classes or the two components $\C_+$ and $\C_-$ of the Clifford algebra. 
Therefore, one of the two components of $\C(A,\sigma)$ splits if and only if one of the two classes associated to $(A,\sigma)$ in $H^1(F,\PGO_{2n})$ 
lifts to $H^1(F,\Spin^+_{2n})$, and this proves the first assertion. 

When this hold, we let $\Spin^+(A,\sigma)$ be the relevant twisted form of $\Spin_{2n}^+$. The same argument as in the split case now shows that the non trivial indices of the $J$-invariant of $\Spin^+(A,\sigma)$ are equal to the non trivial indices of $J(A,\sigma)$, and again the only difference 
comes from the presentations of $\Ch^*(\Spin^+_{2n})$ and $\Ch^*(\PGO_{2n})$: one of the two generators $x_2$ and $x_1+x_2$ is killed by 
the pull-back map $\Ch^1(\PGO_{2n})\mapsto\Ch^1(\Spin^+_{2n})$ by~\ref{excharmap}. 
Hence the second index in $J(A,\sigma)$ is zero, and this finishes the proof. 
\end{proof}
If $(A,\sigma)$ is half-spin, we can refine the inequalities 
given in~\ref{even.cor} by applying Theorem~\ref{corj0} to the half-spin group 
$\Spin^+(A,\sigma)$. We get the following: 
\begin{cor}
\label{hs.cor}
Assume $\deg(A)\equiv 0[4]$ and $(A,\sigma)$ is half-spin, that is its Clifford algebra has a split component. 
The following hold: 
\begin{enumerate} 
\item If $\ii_A>0$, then $j_1>0$. 
\item If $\ii_A>1$, then $j_1>1$. 
\end{enumerate}
\end{cor}
Indeed, in degree $1$,  the Chow group modulo $2$ of a half-spin group is
\[\Ch^1(\Spin^+_{2n})=\Ch^1(\BX)/\im(\cc^{(1)}_{hs})=\Lambda_\omega/\T\otimes_\zz \mf.\]
It is generated by $x_1=\pi(h_1)=\pi\bigl(c_1(\LL(\omega_1)\bigr)$. 
Hence the common index in this case is $\ii_J=\ii_A$. 
Moreover, if $\ii_A>1$, then $4|\deg(A)$ and $k_1>1$. 

\subsection*{Witt-equivalent algebras with involution}
As for quadratic forms, the $J$-invariant of an algebra with involution only depends on its Witt class, as we now proceed to show. 
Consider two Brauer-equivalent algebras $A$ and $B$, respectively endowed with the orthogonal involutions $\sigma$ and $\tau$. 
They can be represented 
as $({\mathrm{ End}}_D(M),\ad_h)$ and $({\mathrm {End}}_D(M'),\ad_{h'})$, for some hermitian modules $(M,h)$ and $(M',h')$ over a division algebra with orthogonal involution $(D,\ba)$, Brauer-equivalent to $A$. 
The algebras $(A,\sigma)$ and $(B,\tau)$ are said to be Witt-equivalent if the hermitian modules $(M,h)$ and $(M',h')$ are Witt-equivalent. 
If so, $(A,\sigma)$ and $(B,\tau)$ are split hyperbolic over the same fields. 
Hence, the corresponding twisted Borel varieties $\XX_A$ and $\XX_B$ are split over the function field of each other. 
By the comparison lemma~\cite[5.18(iii)]{PSZ}, the indecomposable motives $R_2\bigr(\PGO(A,\sigma)\bigr)$ and 
$R_2\bigr(\PGO(B,\tau)\bigr)$ are isomorphic. 
So they have the same Poincar\'e polynomial, and by~\eqref{poincare}, we get
\begin{prop}
\label{wittai.prop}
Let $(A,\sigma)$ and $(B,\tau)$ be two Witt-equivalent algebras with involution. 
All the non trivial indices in their $J$-invariant are equal. 
\end{prop} 
This proposition will prove useful to complete the classification in degree $8$. 
Before, we study the degree $6$ and degree $4$ cases. 
\subsection*{Classification results in degree $6$} 
Assume that $\deg(A)\equiv 2[4]$. 
From the table~\cite[4.13]{PSZ}, $6$ is the smallest value for which the $J$-invariant may be non trivial. 
This is not surprising, since in degree $2$, any quaternion algebra endowed with an orthogonal involution 
with trivial discriminant is split hyperbolic~\cite[(7.4)]{BI}. 
In degree $6$, the $J$-invariant is given by $J(A,\sigma)=(j_1)$, with $0\leq j_1\leq 2$. 
It can be computed as follows: 
\begin{prop}
 Let $A$ be a degree $6$ algebra endowed with an orthogonal involution $\sigma$ with trivial discriminant. 
Its $J$-invariant is given by 
\[J(A,\sigma)=(\ii_S),\]
where, as before, $\ii_S=\ii_+=\ii_-$ is the 
$2$-adic valuation of any component of the Clifford algebra $\C(A,\sigma)$. 
So, we have: 
\begin{enumerate} 
\item $J=(0)\iff (A,\sigma)$ is split hyperbolic. 
\item $J=(1)\iff (A,\sigma)$ is split isotropic and non hyperbolic. 
\item $J=(2)\iff(A,\sigma)$ is anisotropic. 
\end{enumerate}
\end{prop}
Note that the algebra can be split or non split in the last case.
 
\begin{proof} 
For degree reasons, the index $\ii_s$ is bounded by $2$. 
Hence the equality $J(A,\sigma)=(\ii_S)$ is a direct consequence of corollary~\ref{odd.cor}. 
Moreover, the fundamental relations~\ref{FR} imply that if $\ii_S\leq 1$, the algebra $A$ is split. If so, $\sigma$ is adjoint to a quadratic form $\varphi$, 
and the $J$-invariant of $(A,\sigma)$ is $J(A,\sigma)=J(\varphi)$ (see~\ref{split.prop}). 
By~\ref{qf6}, this proves (1) and (2), and also (3) in the split case. 
To finish the proof, it is enough to check that if $A$ is non split, then $\sigma$ is anisotropic. 
For the sake of contradiction, assume $A$ is non split, that is $A=M_3(Q)$ for some division quaternion algebra $Q$ over $k$, and 
$\sigma$ is isotropic. Since $A$ is the endomorphism ring of a $3$-dimensional $Q$-module, $(A,\sigma)$ is a hyperbolic extension, in the sense of~\cite[3.1]{Ga:01}, of $(Q,\sigma_{\mathrm {an}})$, for some 
orthogonal involution $\sigma_{\mathrm{an}}$ of $Q$. Moreover, by~\cite[(7.5)]{BI}, we have $d(\sigma_{\mathrm {an}})=d(\sigma)=1$. 
This is impossible if $Q$ is non split (see~\cite[(7.4)]{BI}). 
\end{proof}

\subsection*{Classification in degree $4$}
Assume the degree of $A$ satisfies $\deg(A)\equiv0[4]$. 
The first two parameters of $J(A,\sigma)=(j_1,j_2,\dots,j_r)$ correspond to generators of degree $1$, and we now have: 
\begin{lem}Assume the algebra $A$ has degree $\deg(A)\equiv 0[4]$, and consider the first indices $j_1$ and $j_2$ of its $J$-invariant. 
We have: 
\begin{enumerate} 
 \item $j_1=0\iff A$ is split; 
\item $j_2=0\iff (A,\sigma)$ is half-spin, i.e. its Clifford algebra $\C(A,\sigma)$ has a split component. 
\end{enumerate}
\end{lem}
\begin{proof} 
We already know from the inequalities~\ref{even.cor}, or alternately from Propositions~\ref{split.prop} and~\ref{hs.prop}, that $j_1=0$ if $A$ is split and $j_2=0$ if $(A,\sigma)$ is half-spin. 
To prove the converse, assume first that $A$ is non split, that is $\ii_A>0$. If $\min\{\ii_+,\ii_-\}>0$, Corollary~\ref{even.cor}(3) shows that $j_1>0$. 
Otherwise, we are in the half spin case, so we can apply Corollary~\ref{hs.cor}, which also gives $j_1>0$. 
Similarly, assume that $\ii_+>0$, and $\ii_->0$. 
If $\ii_A>0$, Corollary~\ref{even.cor}(3) gives $j_2>0$. 
If $A$ is split, Corollary~\ref{qf.cor} gives the conclusion since, by Proposition~\ref{split.prop}, $j_2$ is the first index of the $J$-invariant of the underlying quadratic form. 
\end{proof}

For algebras of degree $4$, the $J$-invariant is given by $J(A,\sigma)=(j_1,j_2)$ with $0\leq j_1,j_2\leq 1$. 
Hence the previous lemma suffices to determine $J(A,\sigma)$. 
Moreover, it is well-known that $(A,\sigma)$ is hyperbolic if and only if one component of $\C(A,\sigma)$ is split (see~\cite[(15.14)]{BI}). 
So we can rephrase the result as follows: 
\begin{lem} 
\label{deg4.lem}
Let $A$ be a degree $4$ algebra endowed with an orthogonal involution $\sigma$ with trivial discriminant. 
Its $J$-invariant $J(A,\sigma)=(j_1,j_2)$ can be computed as follows: 
\begin{enumerate} 
\item $j_1$ is $0$ if $A$ is split and $1$ otherwise; 
\item $j_2$ is $0$ if $\sigma$ is hyperbolic, and $1$ otherwise. 
\end{enumerate}
\end{lem}
We can give a precise description of $(A,\sigma)$ is each case. 
As explained in~\cite[(15.14)]{BI}, since $A$ has degree $4$ and $\sigma$ has trivial discriminant, $(A,\sigma)$ has a canonical decomposition 
$(A,\sigma)=(Q_1,\ba)\otimes_F(Q_2,\ba)$, where the quaternion algebras $Q_1$ and $Q_2$ are the two components of 
the Clifford algebra $\C(A,\sigma)$, each endowed with its canonical involution. 
The algebra $A$ is split if and only if $Q_1$ and $Q_2$ are isomorphic, in which case $\sigma$ is adjoint to the norm form of 
$Q_1=Q_2$, which is a $2$-fold Pfister form. In this case, $J=(0,0)$ if this form is hyperbolic and $J=(0,1)$ if it is anisotropic. 
If $A$ is non split, then $Q_1$ and $Q_2$ are not isomorphic. If one of them, say $Q_1$ is split, then $A$ has index $2$, $A=M_2(Q_2)$, 
$\sigma$ is hyperbolic, and $J=(1,0)$. Otherwise, $J=(1,1)$, the involution is anisotropic, and $A$ has index $2$ or $4$. 
Using this description, one may construct explicit examples for each possible value of the $J$-invariant. 


\section{The trialitarian case}\label{tricase}

From now on, we assume that $(A,\sigma)$ has degree $8$. 
The $J$-invariant of $(A,\sigma)$ is a triple $J(A,\sigma)=(j_1,j_2,j_3)$ with $0\leq j_1,j_2\leq 2$ and $0\leq j_3\leq 1$. 
In this section, we will explain how to compute $J(A,\sigma)$. As a consequence of our results, we will prove: 
\begin{cor} 
\label{values.cor} 
(i) There is no algebra of degree $8$ with orthogonal involution with trivial discriminant having $J$-invariant equal to $(1,2,0)$, $(2,1,0)$ or $(2,2,0)$. 

(ii) All other possible values do occur. 
\end{cor}
In particular, this shows that the restrictions described in the table~\cite[4.13]{PSZ} (see also \S\ref{appendix}), which were obtained by applying the Steenrod operations on $\Ch^*(G_0)$ ({\em loc. cit.} 4.12) are not the only ones. 

Recall that the group $\PGO(A,\sigma)$ is of type $\D_4$. 
To complete the classification in this case, we need to understand the action of the symmetric group $S_3$ on the $J$-invariant (see~\ref{triality.rem}). 
Let $(B,\tau)$ and $(C,\gamma)$ be the two components of the Clifford algebra $\C(A,\sigma)$, each endowed with its canonical involution. 
It follows from the structure theorems~\cite[(8.10) and(8.12)]{BI} that both are degree $8$ algebras with orthogonal involutions. 
The triple $\bigl((A,\sigma),(B,\tau),(C,\gamma)\bigr)$ is a trialitarian triple in the sense of {\em loc.cit.} \S42.A, and in particular, 
the Clifford algebra of any of those three algebras with involution is the direct product of the other $2$. 
Hence, if one of them, say $(A,\sigma)$ is split, then the other two are half-spin. 
\begin{dfn}
 The trialitarian triple $\bigl((A,\sigma),(B,\tau),(C,\gamma)\bigr)$ is said to be ordered by indices if the indices of the algebras $A$, $B$ and $C$ satisfy
\[\ind(A)\leq\ind(B)\leq\ind(C).\]
\end{dfn}
The $J$-invariant of such a triple can be computed as follows:
\begin{thm}
\label{trial.thm}
Let $\bigl((A,\sigma),(B,\tau),(C,\gamma)\bigr)$ be a trialitarian triple ordered by indices, so that 
$\ii_A\leq\ii_B\leq\ii_C$. 
The $J$-invariants are given by 
\[
J(A,\sigma)=(j,j',j_3)\mbox{\ \  and\ \  }
J(B,\tau)=J(C,\gamma)=(j',j,j_3), 
\]
where $j=\min\{\ii_A,2\}$ and $j'=\min\{\ii_B,\ii_C,2\}=\min\{\ii_B,2\}$. 

Moreover, the third index $j_3$ is $0$ if the involution is isotropic and $1$ otherwise.  
\end{thm}
\begin{rem}

(i) The first index of the $J$-invariant of a degree $8$ algebra with involution $(D,\rho)$ 
is $\min\{\ii_D,2\}$ if $D$ is not of maximal index in its triple. But  it might be strictly smaller in general. 
In~\ref{111.ex} below, we will give an explicit example where $j_1<\ii_D=2$. 

(ii) By~\ref{triality.rem}, we already know that $j_3$ does not depend on the choice of an element of the triple. 
On the other hand, as explained in~\cite{Ga:trial}, the involutions $\sigma$, $\tau$ and $\gamma$ are either all isotropic or all anisotropic. 
The triple is said to be isotropic or anisotropic accordingly. 

\end{rem}
\begin{proof} 
To start with, let us compute the first two indices $j_1$ and $j_2$ of the $J$-invariant of $(A,\sigma)$. 
Since we are in degree $8$, they are both bounded by $2$. 
Moreover, the triple being ordered by indices, the common index is given by $\ii_J=\ii_A$. 
So the equality $j_1=j$ follows directly from the inequalities of Corollary~\ref{even.cor}. 
If additionally $j'=j$,  the very same argument gives $j_2=j'$. 
Assume now that $j$ and $j'$ are different, that is $j<j'$. 
If so, $j=0$ or $j=1$. In the first case, we have $\ii_A=0$ so that the algebra $A$ is split, and the result follows from~\ref{qf8} and 
~\ref{split.prop}. 
The only remaining case is $j=\ii_A=1$ and $\ii_B\geq 2$, so that $j'=2$. 
Consider the function field $F_A$ of the Severi-Brauer variety of $A$, which is a generic splitting field of $A$. 
By the fundamental relations~\ref{FR}, the algebra $C$ is Brauer equivalent to $A\otimes B$. 
Hence Merkurjev's index reduction formula~\cite{Mer} says 
\[\ind(B_{F_A})=\min\{\ind(B),\ind(B \otimes A)\}=\ind(B).\]
So the values of $\ii_B$ and $j'$ are the same over $F$ and $F_A$. 
We know the result holds over $F_A$ by reduction to the split case.  
Since the index $j_2$ can only decrease under scalar extension, 
we get $j_2\geq j'=2$, which concludes the proof in this case. 

So the $J$-invariant of $(A,\sigma)$ is given by 
$J(A,\sigma)=(j,j',j_3)$ for some integer $j_3$. 
Let us now compute the $J$-invariant of $(B,\tau)$ and $(C,\gamma)$. 
Recall from~\ref{triality.rem} that 
$(j,j',j_3)$ and $(j',j,j_3)$ are the only possible values.  
So, if $j=j'$, there is no choice and we are done. 
Again, there are two remaining cases. Assume first that $j=\ii_A=0$ and $j'\geq 1$, so that $J(A,\sigma)=(0,j',j_3)$. 
Since $A$ is split, $(B,\tau)$ and $(C,\gamma)$ are half-spin, so they have trivial $j_2$ and this gives the result.  
Assume now that $j=1$ and $j'=2$, so that $J(A,\sigma)=(1,2,j_3)$. 
By the previous case, over the field $F_A$, both $(B,\tau)$ and $(C,\gamma)$ have $J$-invariant 
$(2,0,j_3)$. So the value over $F$ has to be $(2,1,j_3)$. 

To conclude the proof, it only remains to compute $j_3$. 
If $A$ is split, this was done in~\ref{qf8}. 
In the anisotropic case, we can reduce to the split case by generic splitting. 
Indeed, by~\cite{Kar:div} in the division case, \cite[Prop.~3]{Siv} in index $4$, and \cite[Cor.~3.4]{PSS} in index $2$ (see also~\cite{Kar:anis}) 
the triple remains anisotropic after scalar extension to a generic splitting field $F_A$ of the algebra $A$. 
Hence $j_3$ is equal to $1$ over $F_A$, and this implies $j_3=1$. 
In the isotropic case, if $\ii_C\geq 2$, then we actually are in the split case. 
Indeed, if $\ind(C)\geq 4$ and $\gamma$ is isotropic, then $C=M_2(D)$ for some degree $4$ division algebra $D$, and $\gamma$ has to be hyperbolic. 
So by~\ref{hypinv}, $(C,\gamma)$ is half-spin, that is $A$ is split. 
The only remaining case is $\ii_A=\ii_B=\ii_C=1$ and all three involutions are isotropic. 
In this case, $(A,\sigma)$ is Witt-equivalent to a non-split algebra of degree $4$ with anisotropic involution, which has $J$-invariant $(1,1)$ by~\ref{deg4.lem}.
Hence, in view of~\ref{wittai.prop},
the $J$-invariant of $(A,\sigma)$, which already has $j_1=j_2=1$ must have $j_3=0$.  
\end{proof}

The first part of Corollary~\ref{values.cor} follows from Theorem~\ref{trial.thm}. 
Indeed, if one of $j_1$, $j_2$ is $2$ and the other one is $\geq 1$, then 
the algebras $A$, $B$ and $C$ are all three non split, and $B$ and $C$ have index $\geq 4$. 
By~\ref{hypinv}, since $A$ and $B$ are non split, the involution $\gamma$ on $C$ is not hyperbolic, so it is anisotropic, and the theorem gives $j_3=1$. 

\subsection*{Explicit examples} 
We now prove the second part of Corollary~\ref{values.cor}. 
Recall from~\ref{split.prop} that if $A$ is split, and $\sigma$ is adjoint to a quadratic form $\varphi$, then 
$J(A,\sigma)=(0,J(\varphi))$. 
Hence any triple with $j_1=0$ is obtained for a suitable choice of $\varphi$ by~\ref{qf8}. 
Considering the components of the even Clifford algebra of those quadratic forms, we also obtain all triples with $j_2=0$ by Theorem~\ref{trial.thm}. 
The maximal value $(2,2,1)$ is obtained from a generic cocycle; such a cocycle exists by~\cite[Thm. 6.4(ii)]{KM06}.  
Hence, it only remains to prove that the values $(1,1,0)$, $(1,1,1)$, $(1,2,1)$ and $(2,1,1)$ occur. For any of those, we will produce an explicit example, 
inspired by the trialitarian triple constructed in~\cite[Lemma 6.2]{QT}

Our construction uses the notion of direct sum for algebras with involution, which was introduced by Dejaiffe~\cite{Dej}.  
Consider two algebras with involution $(E_1,\theta_1)$ and $(E_2,\theta_2)$ which are Morita-equivalent, 
that is $E_1$ and $E_2$ are Brauer equivalent and the involutions $\theta_1$ and $\theta_2$ are of the same type. 
Dejaiffe defined a notion of Morita equivalence data, and explains how to associate to any such data an algebra with involution 
$(A,\sigma)$, which is called a direct sum of $(E_1,\theta_1)$ and $(E_2,\theta_2)$. 
In the split orthogonal case, if $\theta_1$ and $\theta_2$ are respectively adjoint to the quadratic forms $\varphi_1$ and $\varphi_2$, 
any direct sum of $(E_1,\theta_1)$ and $(E_2,\theta_2)$ is adjoint to $\varphi_1\oplus\qform{\lambda}\varphi_2$ for
some $\lambda\in F^\times$, and the choice of a Morita-equivalence data precisely amounts to the choice of a scalar $\lambda$. In general, as the split case shows, there exist non isomorphic direct sums of
two given algebras with involution. We will use the following characterization of direct sums~\cite[Lemma 6.3]{QT}:
\begin{lem}
The algebra with involution $(A,\sigma)$ is a direct sum of $(E_1,\theta_1)$ and $(E_2,\theta_2)$ 
if and only if there is an embedding of the direct product $(E_1,\theta_1)\times(E_2,\theta_2)$ in $(A,\sigma)$ and 
$\deg(A)=\deg(E_1)+\deg(E_2)$. 
\end{lem}
Slightly extending Garibaldi's `orthogonal sum lemma'~\cite[Lemma 3.2]{Ga:01}, we get: 
\begin{prop}
\label{triple.prop}
Let $Q_1,Q_2,Q_3$ and $Q_4$ be quaternion algebras such that $Q_1\otimes Q_2$ and $Q_3\otimes Q_4$ are Brauer equivalent. 
If $(A,\sigma)$ is a direct sum of $(Q_1,\ba)\otimes(Q_2,\ba)$ and $(Q_3,\ba)\otimes(Q_4,\ba)$ 
then one of the two components of the Clifford algebra of $(A,\sigma)$ is a direct sum of 
$(Q_1,\ba)\otimes(Q_3,\ba)$ and $(Q_2,\ba)\otimes(Q_4,\ba)$, while the other is a direct sum of 
$(Q_1,\ba)\otimes(Q_4,\ba)$ and $(Q_2,\ba)\otimes(Q_3,\ba)$.
\end{prop}
\begin{rem}
If one of the four quaternion algebras is split, as we assumed in~\cite{QT}, then all three direct sums have a hyperbolic component. 
Hence they are uniquely defined. This is not the case anymore in the more general setting considered here. 
The algebra with involution $(A,\sigma)$ does depend on the choice of an equivalence data. 
Nevertheless, once such a choice is made, its Clifford algebra is well defined. So the equivalence data defining the other two direct sums are determined by the one we have chosen. 
\end{rem}
\begin{proof} 
Denote $(E_1,\theta_1)=(Q_1,\ba)\otimes(Q_2,\ba)$ and 
$(E_2,\theta_2)=(Q_3,\ba)\otimes(Q_4,\ba)$. By~\cite[(15.12)]{BI}, their Clifford algebras with canonical involution are 
$(Q_1,\ba)\times(Q_2,\ba)$, and $(Q_3,\ba)\times(Q_4,\ba)$ respectively. 
The embedding of the direct product $(E_1,\theta_1)\times(E_2,\theta_2)$ in $(A,\sigma)$ 
induces an embedding of the tensor product of their Clifford algebras in the Clifford algebra of $(A,\sigma)$: 
\[\bigl( (Q_1,\ba)\times(Q_2,\ba)\bigr)\otimes\bigl((Q_3,\ba)\times(Q_4,\ba)\bigr)\hookrightarrow   (\C(A,\sigma),\underline\sigma).\] 
This tensor product splits as a direct product of four tensor products of quaternion algebras with canonical involution; for 
degree reasons, two of them embed in each component of $\C(A,\sigma)$. 
To identify them, it is enough to look at their Brauer classes. From the hypothesis, we have Brauer equivalences 
$Q_1\otimes Q_3\sim Q_2\otimes Q_4$ and $Q_1\otimes Q_4\sim Q_2\otimes Q_3$. 
If $Q_1\otimes Q_3$ and $Q_1\otimes Q_4$ are not Brauer equivalent, that is if $A$ is non split, this concludes the 
proof. Otherwise, all four tensor products are isomorphic, and the result is still valid. 
\end{proof}

With this in hand, we now give explicit examples of algebras with involution having $J$-invariant $(1,2,1)$, $(2,1,1)$, $(1,1,1)$ and $(1,1,0)$. 
\begin{ex}
\label{12.ex}
Let $F=K(x,y,z,t)$ be a function field in $4$ variables over a field $K$, and consider the following quaternion algebras over $F$: 
\[Q_1=(x,zt),\ Q_2=(y,zt),\ Q_3=(xy,z)\mbox{ and }Q_4= (xy,t). 
\]
We let $(A,\sigma)$ be a direct sum of $(Q_1,\ba)\otimes(Q_2,\ba)$ and $(Q_3,\ba)\otimes(Q_4,\ba)$ as in~\ref{triple.prop}, 
and denote by $(B,\tau)$, and respectively $(C,\gamma)$, the component of $\C(A,\sigma)$ Brauer equivalent to 
$Q_1\otimes Q_3\sim (x,t)\otimes (y,z)$ and $Q_1\otimes Q_4\sim (x,z)\otimes(y,t)$. 
The algebras $A$, $B$ and $C$ have index $2$, $4$ and $4$, so that $\bigl((A,\sigma),(B,\tau),(C,\gamma)\bigl)$ is a trialitarian triple ordered by indices. 
By Theorem~\ref{trial.thm}, we get $J(A,\sigma)=(1,2,j_3)$ and $J(B,\tau)=J(C,\gamma)=(2,1,j_3)$ for some $j_3$. 
Finally, assertion (i) of Corollary~\ref{values.cor} implies $j_3=1$; in other words, this triple is anisotropic. 
\end{ex}

\begin{ex}
\label{111.ex}
This example is obtained from the previous one by scalar extension. 
Consider the Albert form $\varphi=\qform{x,t,-xt,-y,-z,yz}$ associated to the biquaternion algebra $Q_1\otimes Q_3$. We let $F'$ be its function field, $F'=F(\varphi)$, and denote by $(A',\sigma')$, $(B',\tau')$ and $(C',\gamma')$ the extensions of $(A,\sigma)$, $(B,\tau)$ and $(C,\gamma)$ to $F'$. 
Since $B$ is Brauer equivalent to $Q_1\otimes Q_3$, the algebra $B'$ has index $2$. 
On the other hand, it follows from Merkurjev's index reduction formula~\cite[Thm.~3]{Mer} that the indices of $A$ and $C$ are preserved by scalar extension to $F'$, so that 
$A'$ and $C'$ have indices $2$ and $4$ respectively. 
Hence $\bigl((A',\sigma'),(B',\tau'),(C',\gamma')\bigl)$ again is a trialitarian triple ordered by indices
and Theorem~\ref{trial.thm} now gives $J(A',\sigma')=J(B',\tau')=J(C',\gamma')=(1,1,j_3)$ for some $j_3$. 
The same argument as in the proof of the first assertion of Corollary~\ref{values.cor} applies here: since $A'$ and $B'$ are non split and $C'$ has index $4$, the involutions are anisotropic 
and Theorem~\ref{trial.thm} gives $j_3=1$. 
Note that, in particular, we have $J(C',\gamma')=(1,1,1)$, even though $C'$ has index $4=2^2$.
\end{ex}

\begin{ex}
We now produce another example of an anisotropic trialitarian triple having $J$-invariant $(1,1,1)$ in which all three algebras have index $2$. 
Namely, consider the $F$-quaternion algebras 
\[Q_1=(x,y),\ Q_2=(x,z),\ Q_3=(x,t)\mbox{ and }Q_4=(x,yzt).\]
Pick an arbitrary orthogonal involution $\rho$ on $H=(x,yz)$ over $F$. 
Since $Q_1\otimes Q_2$ is isomorphic to $2$ by $2$ matrices over $H$, the tensor product of the canonical involutions 
of $Q_1$ and $Q_2$ is adjoint to a $2$-dimensional hermitian form $h_{12}$ over $(H,\rho)$. 
Similarly, $(Q_3,\ba)\otimes (Q_4,\ba)$ is isomorphic to $M_2(H)$ endowed with the adjoint involution with respect to some 
hermitian form $h_{34}$. 
Since $h_{12}$ and $h_{34}$ are both anisotropic, the hermitian form $h=h_{12}\oplus\qform{u}h_{34}$ over $H''=H\otimes F(u)$, for some indeterminate $u$, 
also is anisotropic. 
We define 
\[(A,\sigma)=(M_4(H''),\ad_h).\]
It is clear from the definition that $(A,\sigma)$ is a direct sum of $(Q_1,\ba)\otimes(Q_2,\ba)$ and $(Q_3,\ba)\otimes(Q_4,\ba)$. 
Hence,  by~\ref{triple.prop}, the two components $(B,\tau)$ and $(C,\gamma)$ of its Clifford algebra 
are Brauer equivalent to $(x,yt)$ and $(x,zt)$. This shows that all three algebras have index $2$. 
Since the involutions are anisotropic, by Theorem~\ref{trial.thm}, their $J$-invariant is $(1,1,1)$. 
\end{ex}

\begin{rem}
Note that there are many other examples, and not all of them can be described as in~\ref{triple.prop}. 
In particular, any triple which includes a division algebra cannot be obtained from this proposition. 
Consider for instance the algebra with involution $(A,\sigma)$ described in~\cite[Example 3.6]{QT02}, and let 
$(B,\tau)$ and $(C,\gamma)$ be the two components of its Clifford algebra. 
As explained there, $A$ is a division indecomposable algebra, and one component of its Clifford algebra, say $B$, has index $2$. 
Since $A$ is Brauer equivalent to $B\otimes C$, its indecomposability guarantees that $C$ is division, and we get 
$J(A,\sigma)=J(C,\gamma)=(2,1,1)$ and $J(B,\tau)=(1,2,1)$.
\end{rem}

To produce examples of algebras with involution having $J$-invariant $(1,1,0)$, we now construct examples of isotropic non split and non half-spin triples. 
As opposed to the previous examples, they can always be described using Proposition~\ref{triple.prop}. Indeed, we get the following explicit description for such triples (cf. Garibaldi's~\cite[Thm.~0.1]{Ga:iso}) 
\begin{prop}
\label{iso.prop}
If $\bigl((A,\sigma),(B,\tau),(C,\gamma)\bigl)$ is an isotropic trialitarian triple with 
$A$, $B$ and $C$ non split, then there exists division quaternion algebras $Q_1$, $Q_2$ and $Q_3$ 
such that $Q_1\otimes Q_2\otimes Q_3$ is split and the triple is described as in~\ref{triple.prop} with $Q_4=M_2(k)$. 
\end{prop}
\begin{proof}
Since $B$ and $C$ are non split, the involution $\sigma$ is not hyperbolic by~\ref{hypinv}. 
Hence $A$ has index $2$, $A=M_4(Q_1)$ for some quaternion algebra $Q_1$ over $k$. 
Fix an orthogonal involution $\rho_1$ on $Q_1$; the involution $\sigma$ is adjoint to a hermitian form 
$h=h_0\oplus h_1$ over $(Q_1,\rho_1)$, with $h_0$ hyperbolic, $h_1$ anisotropic and both of dimension $2$ and trivial discriminant. 
Therefore, $(A,\sigma)$ is a direct sum of $(M_2(Q_1),\ad_{h_0})$ and $(M_2(Q_1),\ad_{h_1})$. 
Since the first summand is hyperbolic, it is isomorphic to 
$(M_2(k),\ba)\otimes(Q_1,\ba)$. The second is $(Q_2,\ba)\otimes(Q_3,\ba)$, where $Q_2$ and $Q_3$ are the two components of the Clifford algebra $\ad_{h_1}$, and 
this concludes the proof.  
\end{proof}

We refer the reader to~\cite[\S 6]{QT} for a more precise description of those triples. They are the only ones for which the $J$-invariant is 
$(1,1,0)$. 

\section{Generic properties}
\label{generic.section}

In the present section we investigate the relationship between the values of the $J$-invariant
of an algebra with involution $(A,\sigma)$ and the $J$-invariant of the respective adjoint quadratic form $\varphi_\sigma$
over the function field $F_A$ of the Severi-Brauer variety of $A$, which is a generic splitting field of $A$. 

\begin{dfn}
We say $(A,\sigma)$ is generically Pfister if $\varphi_\sigma$ is a Pfister form.
Observe that in this case $\deg A$ is always a power of $2$ and the
$J$-invariant over $F_A$ has the form:
$$
J\bigl((A,\sigma)_{F_A}\bigr)=(0,\ldots,0,*)
$$
(all zeros except possibly the last entry which is $0$ or $1$).

We say $(A,\sigma)$ is in $I^s$, $s> 2$,
if $\varphi_\sigma$ belongs to the $s$-th power $I^s(F_A)$ of the fundamental ideal $I(F_A)\subset W(F_A)$ of the Witt ring of $F_A$. 
\end{dfn}

\begin{thm}\label{Meta-conj}
Let $(A,\sigma)$ be an algebra of degree $2n$ with orthogonal involution with trivial discriminant.
\begin{itemize}
\item[(a)]
If $(A,\sigma)$ is in $I^s$, $s>2$, then 
$$J(A,\sigma)=(j_1,\underbrace{0,\ldots,0,}_{2^{s-2}-1\text{ times}}*,\ldots,*)$$ 
\item[(b)]
In particular, if $(A,\sigma)$ is generically Pfister, then $J(A,\sigma)=(*,0,\ldots,0,*)$.  
\end{itemize}
\end{thm}
\begin{proof}
(a) Let $X=\D_n/P_i$ be the variety of maximal parabolic subgroups of type
$i:=2\cdot[\frac{n+1}{2}]-2^{s-1}+1$ (For parabolic subgroups we use notation from \cite[2.1]{PS09}). Since $i$ is odd, $A_{F(X)}$ splits,
and therefore the quadratic form $\varphi_\sigma$ is defined over $F(X)$.
By assumption $\varphi_\sigma\in I^s\bigl(F(X)\bigr)$. The Witt index of $\varphi_\sigma$ is at least $i$.
Therefore the anisotropic part of $\varphi_\sigma$ has dimension at most $2(n-i)<2^s$.
Thus, by the Arason-Pfister theorem $\varphi_\sigma$ is hyperbolic. In particular,
the variety $X$ is generically split. Therefore by \cite[Theorem~2.3]{PS09b} we obtain the desired expression for
the $J$-invariant.

(b) Finally, if $(A,\sigma)$ is generically Pfister, then $\varphi_\sigma\in I^s\bigl(F(X)\bigr)$,
where $2^s=2n$ and (b) follows from (a).
\end{proof}

\begin{rem}
Let $(j_2,\ldots,j_r)$ be the $J$-invariant of $\varphi_\sigma$ over $F_A$, $r=[\frac{n+2}{2}]$.
In view of the theorem one may conjecture
that the $J$-invariant of $(A,\sigma)$ is obtained from $J(\varphi_\sigma)$ just
by adding an arbitrary left term, i.e.
$$
J(A,\sigma)=(*,j_2,\ldots,j_r).
$$
For example, 
if $\varphi_\sigma$ is excellent, then the $J$-invariant should be equal to
$$J(A,\sigma)=(*,0,\ldots,0,*,0,\ldots,0),$$ 
where
the second $*$ has degree $2^s-1$ for some $s$ and equals either $0$ or $1$.

By the results of~\S\ref{tricase}, observe that this holds for algebras of degree $8$.
\end{rem}


\section{Appendix}
\label{appendix}
The following table provides the values of the parameters 
of the $J$-invariant for all orthogonal groups (here $p=2$).

\medbreak

\medbreak

\noindent
\begin{tabular}{|l|l|l|l|l|}
\hline
$G_0$  & $r$ & $d_i$ & $k_i$ & restrictions on $j_i$\\\hline

$\SO_n$  & $[\frac{n+1}{4}]$ & $2i-1$ & 
$[\log_2\frac{n-1}{d_i}]$ & 
if $d_i+l=2^sd_m$\\
&  & & & and $2\nmid \binom{d_i}{l}$, then $j_m\le j_i+s$ \\
\hline
$\Spin^{\pm}_{2n},\,2\mid n$&$\frac n2$&$1,\,i=1$&$2^{k_1}\parallel n$&the same restrictions\\
&& $2i-1,\,i\ge 2$ &$[\log_2\frac{2n-1}{d_i}]$ &\\
\hline
$\Spin_n$&$[\frac{n-3}{4}]$&$2i+1$&$[\log_2\frac{n-1}{d_i}]$& the same restrictions\\
&&&&\\
\hline
$\PGO_{2n}$  & $[\frac{n+2}{2}]$ & $1,\,i=1$ & $2^{k_1}\parallel n$ & the same restrictions\\
&& $2i-3,\,i\ge 2$ & $[\log_2\frac{2n-1}{d_i}]$& assuming $i,m\ge 2$\\
\hline
\end{tabular}

\medbreak

\medbreak

Note that this table coincides with \cite[Table~4.13]{PSZ} except
of the last column which in our case contains more restrictive conditions.
For $s=0$ and $1$ the restrictions in the last column are equivalent to those
in \cite[Table~4.13]{PSZ}.

The conditions of the last column are simply translation of \cite[Prop.~5.12]{Vi05}
from the language of Vishik's $J$-invariant to ours.

All values of the $J$-invariant which satisfy the restrictions given
in the table are called admissible.

\medbreak

\noindent{{\textbf{Acknowledgments.}  
The authors are sincerely grateful to Victor Petrov for numerous discussions and,
in particular, for providing the table in the appendix.

The authors gratefully acknowledge the support of
the SFB/Transregio~45 Bonn-Essen-Mainz.
The second author was supported by the MPIM Bonn.
The last author was supported by NSERC Discovery 385795-2010 and 
Accelerator Supplement 396100-2010 grants. He also thanks Universit\'e Paris 13 
for its hospitality in June 2010.}}

\bibliographystyle{plain}

\end{document}